\documentclass[final]{siamltex}
\usepackage[colorlinks]{hyperref}            
\usepackage{color}   
\usepackage{graphicx}
\usepackage[english]{babel}
\usepackage{amsmath}
\usepackage{amssymb,amsmath}
\usepackage{setspace}
\usepackage{epsfig,psfig}

\definecolor{dullmagenta}{RGB}{102,0,102}   
\definecolor{darkred}{RGB}{178,34,34}   

\hypersetup{urlcolor=dullmagenta
         ,citecolor=dullmagenta
         ,linkcolor=dullmagenta
         ,pagecolor=white
         }

\newcommand{\Argmin}[1]{\mathop{\mbox{argmin}}_{#1}}

\def\CtSteps{T}

\newcommand{\order}{\mathcal{O}}

\newcommand{\BPE}{$(\text{BP}_{\epsilon})$ }
\newcommand{\QPL}{$(\text{QP}_{\lambda})$ }
\newcommand{\LST}{$(\text{LS}_{\tau})$ }

\newcommand{\mmax}[2]{\tiny #1/#2}
\newcommand{\nesta}{NESTA }
\newcommand{\nestansp}{NESTA}

\title{NESTA: A Fast and Accurate First-order Method\\
  for Sparse Recovery}

\author{Stephen Becker, J\'er\^ome Bobin and Emmanuel
J. Cand\`es\thanks{Applied and Computational Mathematics, Caltech,
  Pasadena, CA 91125 ({\tt srbecker, bobin,
    emmanuel@acm.caltech.edu}). This work has been partially
  supported by ONR grants N00014-09-1-0469 and N00014-08-1-0749, by
  a DARPA grant FA8650-08-C-7853, and by the 2006 Waterman Award from
  NSF. Submitted April 16, 2009.}}

\date{\today}

\begin{document}

\maketitle

\begin{abstract}
Accurate signal recovery or image reconstruction from indirect and
possibly undersampled data is a topic of considerable interest; 
for example, the literature in the recent field of compressed
sensing is already quite immense.
Inspired by recent breakthroughs in the development of novel
first-order methods in convex optimization, most notably Nesterov's
smoothing technique, this paper introduces a fast and accurate
algorithm for solving common recovery problems in signal
processing. In the spirit of Nesterov's work,  one of the key
ideas of this algorithm is
a subtle averaging of sequences of
iterates, which has been shown to improve the convergence properties
of standard gradient-descent algorithms. This paper demonstrates
that this approach is ideally suited for solving large-scale
compressed sensing reconstruction problems as 1) it is
computationally efficient, 2) it is accurate and returns solutions
with several correct digits, 3) it is flexible and amenable to many
kinds of reconstruction problems, and 4) it is robust in the sense
that its excellent performance across a wide range of problems does
not depend on the fine tuning of several parameters.  Comprehensive
numerical experiments on realistic signals exhibiting a large
dynamic range show that this algorithm compares favorably with
recently proposed state-of-the-art methods. We also apply the
algorithm to solve other problems for which there are fewer
alternatives, such as total-variation minimization, and convex
programs seeking to minimize the $\ell_1$ norm of $Wx$ under
constraints, in which $W$ is not diagonal.
\end{abstract}



\pagestyle{myheadings}
\thispagestyle{plain}


\section{Introduction}\label{sec:intro}

Compressed sensing (CS) \cite{CRT:cs,CT:cs3,don:cs} is a novel
sampling theory, which is based on the revelation that one can exploit
sparsity or compressibility when acquiring signals of general
interest. In a nutshell, compressed sensing designs nonadaptive
sampling techniques that condense the information in a compressible
signal into a small amount of data. There are some indications that
because of the significant reduction in the number of measurements
needed to recover a signal accurately, engineers are changing the way
they think about signal acquisition in areas ranging from
analog-to-digital conversion \cite{darpa}, digital optics, magnetic
resonance imaging \cite{lustig07}, seismics \cite{LHerm07} and
astronomy \cite{CS_Bobin}.

In this field, a signal $x^0 \in \mathbb{R}^n$ is acquired by collecting
data of the form 
\[
b = Ax^0 + z,
\]
where $x^0$ is the signal of interest (or its coefficient
sequence in a representation where it is assumed to be fairly sparse),
$A$ is a known $m \times n$ ``sampling'' matrix, and $z$ is a noise
term. In compressed sensing and elsewhere, a standard approach
attempts to reconstruct $x^0$ by solving
\begin{equation}\label{eq:cs}
\begin{array}{ll}
  \text{minimize}   & \quad f(x)\\
  \text{subject to} & \quad  \|b - Ax\|_{\ell_2} \le \epsilon,
\end{array}
\end{equation}
where $\epsilon^2$ is an estimated upper bound on the noise power.
The choice of the regularizing function $f$ depends on prior
assumptions about the signal $x^0$ of interest: if $x^0$ is
(approximately) sparse, an appropriate convex function is the $\ell_1$
norm (as advocated by the CS theory); if $x^0$ is a piecewise constant
object, the total-variation norm provides accurate recovery results,
and so on.

Solving large-scale problems such as \eqref{eq:cs} (think of $x^0$ as
having millions of entries as in mega-pixel images) is
challenging. Although one cannot review the vast literature on this
subject, the majority of the algorithms that have been proposed are
unable to solve these problems {\em accurately} with low computational
complexity. On the one hand, standard second-order methods such as
interior-point methods \cite{l1magic,l1ls,pdco} are accurate but
problematic for they need to solve large systems of linear equations
to compute the Newton steps. 
On the other hand, inspired by iterative
thresholding ideas \cite{daubechies:iter,Rest:Nowak,CombettesWajs05},
we have now available a great number of first-order methods, see
\cite{ist:fnw,COS:XXX:08,FPC,FPC2} and the many earlier references
therein, which may be faster but not necessarily accurate. Indeed,
these methods are shown to converge slowly, and typically need a very
large number of iterations when high accuracy is required.  

We would like to pause on the demand for high accuracy since this is
the main motivation of the present paper. While in some applications,
one may be content with one or two digits of accuracy, there are
situations in which this is simply unacceptable. Imagine that the
matrix $A$ models a device giving information about the signal $x^0$,
such as an analog-to-digital converter, for example. Here, the ability
to detect and recover low-power signals that are barely above the
noise floor, and possibly further obscured by large interferers, is
critical to many applications. In mathematical terms, one could have a
superposition of high power signals corresponding to components
$x^0[i]$ of $x^0$ with magnitude of order 1, and low power signals
with amplitudes as far as 100 dB down, corresponding to components
with magnitude about $10^{-5}$. In this regime of high-dynamic range,
very high accuracy is required. In the example above, one would need
at least five digits of precision as otherwise, the low power signals
would go undetected.

Another motivation is solving \eqref{eq:cs} accurately when the signal
$x^0$ is not exactly sparse, but rather approximately sparse, as in the
case of real-world compressible signals.  Since exactly sparse signals
are rarely found in applications---while compressible signals are
ubiquitous---it is important to have an accurate first-order method
to handle realistic signals.

\subsection{Contributions}

A few years ago, Nesterov \cite{SMSN_Nesterov} published a seminal
paper which couples smoothing techniques (see \cite{NemiBook} and the
references therein) with an improved gradient method to derive
first-order methods which achieve a convergence rate he had proved to
be optimal \cite{Neste83} two decades earlier. As a consequence of
this breakthrough, a few recent works have followed up with improved
techniques for some very special problems in signal or image
processing, see \cite{BeckT08,Dahl08,Weiss08,aujol} for example, or
for minimizing composite functions such as $\ell_1$-regularized
least-squares problems \cite{NestCompo}.  In truth, these novel
algorithms demonstrate great promise; they are fast, accurate and
robust in the sense that their performance does not depend on the fine
tuning of various controlling parameters.

This paper also builds upon Nesterov's work by extending some of his 
works discussed just above, and proposes an algorithm---or, better
said, a class of algorithms---for solving recovery problems from
incomplete measurements. We refer to this algorithm as
\nestansp---a shorthand for Nesterov's algorithm---to acknowledge the fact
that it is based on his method. The main purpose and the
contribution of this paper consist in showing that \nesta obeys the
following desirable properties.
\begin{remunerate}
\item {\em Speed:} \nesta is an iterative algorithm where each
iteration is decomposed into three steps, each involving only a few
matrix-vector operations when $A^*A$ is an orthogonal projector and,
more generally, when the eigenvalues of $A^*A$ are well
clustered. This, together with the accelerated convergence rate of
Nesterov's algorithm \cite{SMSN_Nesterov,BeckT08}, makes \nesta a
method of choice for solving large-scale problems. Furthermore,
\nestansp's convergence is mainly driven by a single smoothing
parameter $\mu$ introduced in Section~\ref{sec:nesterov}. One can
use continuation techniques \cite{FPC,FPC2} to dynamically update
this parameter to substantially accelerate this algorithm.

\item {\em Accuracy:} \nesta depends on a few parameters that can be
set in a very natural fashion. In fact, there is a trivial
relationship between the value of these parameters and the desired
accuracy. Furthermore, our numerical experiments demonstrate that
\nesta can find the first 4 or 5 significant digits of the optimal
solution to \eqref{eq:cs}, where $f(x)$ is the $\ell_1$ norm or
the total-variation norm of $x$, in a few hundred iterations.  This
makes \nesta amenable to solve recovery problems involving signals
of very large sizes that also exhibit a great dynamic range.

\item{\em Flexibility:} \nesta can be adapted to solve many problems
beyond $\ell_1$ minimization with the same efficiency, such as
total-variation (TV) minimization problems. In this paper, we will
also discuss applications in which $f$ in \eqref{eq:cs} is given by
$f(x) = \|W x\|_{\ell_1}$, where one may think of $W$ as a
short-time Fourier transform also known as the Gabor transform, a
curvelet transform, an undecimated wavelet transform and so on, or a
combination of these, or a general arbitrary dictionary of
waveforms (note that this class of recovery problems also include
weighted $\ell_1$ methods \cite{Candes:2008le}). This is
particularly interesting because recent work \cite{EladPrior}
suggests the potential advantage of this analysis-based approach over the classical
{\em basis pursuit} in solving important inverse problems
\cite{EladPrior}.
\end{remunerate}

A consequence of these properties is that \nestansp, and more generally
Nesterov's method, may be of interest to researchers working in the
broad area of signal recovery from indirect and/or undersampled data.

Another contribution of this paper is that it also features a fairly
wide range of numerical experiments comparing various methods against
problems involving realistic and challenging data. By challenging, we
mean problems of very large scale where the unknown solution exhibits
a large dynamic range; that is, problems for which classical
second-order methods are too slow, and for which standard first-order
methods do not provide sufficient accuracy. More specifically,
Section~\ref{sec:numeric} presents a comprehensive series of numerical
experiments which illustrate the behavior of several state-of-the-art
methods including interior point methods \cite{l1ls}, projected
gradient techniques \cite{FPC,SPGL,ist:fnw}, fixed point continuation
and iterative thresholding algorithms \cite{FPC,YinCS08,BeckT08}. It
is important to consider that most of these methods have been
perfected after several years of research \cite{l1ls,ist:fnw}, and did
not exist two years ago. For example, the Fixed Point Continuation
method with Active Set \cite{FPC2}, which represents a notable
improvement over existing ideas, was released while we were working on
this paper.

\subsection{Organization of the paper and notations}

As emphasized earlier, \nesta is based on Nesterov's ideas and
Section~\ref{sec:nesterov} gives a brief but essential description of
Nesterov's algorithmic framework. The proposed algorithm is introduced
in Section~\ref{sec:nesta}. Inspired by continuation-like schemes, an
accelerated version of \nesta is described in
Section~\ref{sec:conti}. We report on extensive and comparative
numerical experiments in
Section~\ref{sec:numeric}. Section~\ref{sec:flexible} covers
extensions of \nesta to minimize the $\ell_1$ norm of $Wx$ under data
constraints (Section~\ref{sec:analysis}), and includes realistic
simulations in the field of radar pulse detection and estimation.
Section \ref{sec:tvmin} extends \nesta to solve total-variation
problems and presents numerical experiments which also demonstrate its
remarkable efficiency there as well.  Finally, we conclude with
Section~\ref{sec:discussion} discussing further extensions, which
would address an even wider range of linear inverse
problems. 

{\em Notations.}  Before we begin, it is best to provide a brief
summary of the notations used throughout the paper. As usual, vectors
are written in small letters and matrices in capital letters. The
$i$th entry of a vector $x$ is denoted $x[i]$ and the $(i,j)$th entry
of the matrix $A$ is $A[i,j]$.

It is convenient to introduce some common optimization problems
that will be discussed throughout. Solving sparse reconstruction
problems can be approached via several different equivalent
formulations. In this paper, we particularly emphasize the quadratically
constrained $\ell_1$-minimization problem
\begin{equation}\label{eq:bp}
\begin{array}{lll}
(\text{BP}_{\epsilon}) & \quad     \text{minimize}   & \quad \|x\|_{\ell_1}\\
&  \quad \text{subject to} & \quad  \|b - Ax\|_{\ell_2} \le \epsilon,
\end{array}
\end{equation}
where $\epsilon$ quantifies the uncertainty about the measurements $b$
as in the situation where the measurements are noisy. 
This formulation is often preferred because a reasonable estimate of
$\epsilon$ may be known. A second
frequently discussed approach considers solving this problem in
Lagrangian form, i.e.
\begin{equation}\label{eq:pls}
(\text{QP}_{\lambda}) \quad \text{minimize} \quad \lambda \|x\|_{\ell_1} + \frac{1}{2}  \|b - Ax\|_{\ell_2}^2, 
\end{equation}
and is also known as the basis pursuit denoising problem (BPDN)
\cite{wave:donoho98}. This problem is popular in signal and image
processing because of its loose interpretation as a maximum {\em a
posteriori} estimate in a Bayesian setting. In statistics, the same
problem is more well-known as the lasso \cite{lasso:tib}
\begin{equation}\label{eq:lasso}
\begin{array}{lll}
(\text{LS}_{\tau}) & \quad     \text{minimize}   & \quad \|b - Ax\|_{\ell_2}\\
&  \quad \text{subject to} & \quad  \|x\|_{\ell_1} \le \tau.
\end{array}
\end{equation}
Standard optimization theory \cite{Rockafellar:1970px} asserts that
these three problems are of course equivalent provided that $\epsilon,
\lambda, \tau$ obey some special relationships. With the exception of
the case where the matrix $A$ is orthogonal, this functional
dependence is hard to compute \cite{SPGL}. Because it is usually more
natural to determine an appropriate $\epsilon$ rather than an
appropriate $\lambda$ or $\tau$, the fact that \nesta solves
($\text{BP}_\epsilon$) is a significant advantage. Further, note that
theoretical equivalence of course does not mean that all three
problems are just as easy (or just as hard) to solve. For instance,
the constrained problem ($\text{BP}_\epsilon$) is harder to solve than
($\text{QP}_\lambda$), as discussed in
Section~\ref{sec:constrained}. Therefore, the fact that \nesta turns
out to be competitive with algorithms that only solve
($\text{QP}_\lambda$) is quite remarkable.


\section{Nesterov's method}
\label{sec:nesterov}



\subsection{Minimizing smooth convex functions}
In \cite{NesteBook,Neste83}, Nesterov introduces a subtle algorithm to
minimize any smooth convex function $f$ on the convex set
$\mathcal{Q}_p$, 
\begin{equation}\label{eq:nestmin}
\min_{x \in \mathcal{Q}_p} f(x). 
\end{equation}
We will refer to $\mathcal{Q}_p$ as the primal feasible set. The
function $f$ is assumed to be differentiable and its gradient $\nabla
f(x)$ is Lipschitz and obeys
\begin{equation}
||\nabla f(x) - \nabla f(y) ||_{\ell_2} \leq L\| x - y\|_{\ell_2};
\end{equation}
in short, $L$ is an upper bound on the Lipschitz constant. With these
assumptions, Nesterov's algorithm minimizes $f$ over $\mathcal{Q}_p$
by iteratively estimating three sequences $\{x_k\}$, $\{y_k\}$ and
$\{z_k\}$ while smoothing the feasible set $\mathcal{Q}_p$.  The
algorithm depends on two scalar sequences $\{\alpha_k\}$ and
$\{\tau_k\}$ discussed below, and takes the following form:
\begin{center}
\centering
\vspace{0.25in}
\begin{tabular}{|c|} \hline
\begin{minipage}[hbt]{0.95\linewidth}
\vspace{0.15in}

\textsf{\textbf{Initialize} $x_0$. For $k \ge 0$,}\\

\textsf{1. Compute $\nabla f(x_k)$.}\\

\textsf{2. Compute $y_k$:}\\

\hspace{0.2in} $y_k\! = \!\Argmin{x \in Q_p} \frac{L}{2} \| x - x_k \|_{\ell_2}^2 + \langle \nabla f(x_k) , x-x_k \rangle$. \\

\textsf{3. Compute $z_k$:}\\

\hspace{0.2in} $z_k \! = \! \Argmin{x \in Q_p} \frac{L}{\sigma_p} p_p(x) + \sum_{i=0}^k \alpha_i \langle \nabla f(x_i) , x-x_i \rangle$. \\

\textsf{4. Update $x_k$:}\\

\hspace{0.2in} $x_k \! = \! \tau_k z_k +(1-\tau_k)y_k$. \\

\textsf{\textbf{Stop} when a given criterion is valid.}

\vspace{0.15in}
\end{minipage}
\\\hline
\end{tabular}
\vspace{0.25in}
\end{center}

At step $k$, $y_k$ is the current guess of the optimal solution. If we
only performed the second step of the algorithm with $y_{k-1}$ instead
of $x_k$, we would obtain a standard first-order technique with
convergence rate $\order(1/k)$. 

The novelty is that the sequence $z_k$ ``keeps in mind'' the previous
iterations since Step 3 involves a weighted sum of already computed
gradients.  Another aspect of this step is that---borrowing ideas from
smoothing techniques in optimization \cite{NemiBook}---it makes use of
a {\em prox-function} $p_p(x)$ for the primal feasible set $Q_p$. This
function is strongly convex with parameter $\sigma_p$; assuming that
$p_p(x)$ vanishes at the prox-center $x_p^c = \Argmin{x} p_p(x)$, this
gives
\[
p_p(x) \ge \frac{\sigma_p}{2} \|x - x_p^c\|_{\ell_2}^2. 
\]
The prox-function is usually chosen so that $x_p^c \in \mathcal{Q}_p$,
thus discouraging $z_k$ from moving too far away from the center
$x_p^c$. 

The point $x_k$, at which the gradient of $f$ is evaluated, is a
weighted average between $z_k$ and $y_k$. In truth, this is motivated
by a theoretical analysis \cite{SMSN_Nesterov,Tseng08}, which shows
that if $\alpha_k =1/2(k+1)$ and $\tau_k = 2/(k+3)$, then the
algorithm converges to
\[
x^{\star} = \Argmin{x_\in Q_p} f(x)
\]
with the convergence rate
\begin{equation}\label{eq:conv}
f(y_k) - f(x^{\star}) \leq \frac{4L p_p(x^{\star})}{(k+1)^2 \sigma_p}.  
\end{equation}
This decay is far better than what is achievable via standard
gradient-based optimization techniques since we have an approximation
scaling like $L/k^2$ instead of $L/k$.

\subsection{Minimizing nonsmooth convex functions}
In an innovative paper \cite{SMSN_Nesterov}, Nesterov recently
extended this framework to deal with nonsmooth convex
functions. Assume that $f$ can be written as
\begin{equation}
\label{eq:saddlezero}
f(x) = \max_{u \in \mathcal{Q}_d} \langle u, W x \rangle, 
\end{equation}
where $x \in \mathbb{R}^n$, $u \in \mathbb{R}^p$ and $W \in
\mathbb{R}^{p \times n}$. We will refer to $\mathcal{Q}_d$ as the dual
feasible set, and suppose it is convex.  This assumption holds for all
of the problems of interest in this paper---we will see in
Section~\ref{sec:nesta} that this holds for $\|x\|_{\ell_1}$,
$\|Wx\|_{\ell_1}$, the total-variation norm and, in general, for any
induced norm---yet it provides enough information beyond the black-box
model to allow cleverly-designed methods with a convergence rate
scaling like $\mathcal{O}(1/k^2)$ rather than
$\mathcal{O}(1/\sqrt{k})$, in the number of steps $k$.

With this formulation, the minimization \eqref{eq:nestmin} can be
recast as the following saddle point problem:
\begin{equation}
\label{eq:saddle}
\min_{x \in \mathcal{Q}_p} \max_{u \in \mathcal{Q}_d} \langle u,Wx \rangle. 
\end{equation}
The point is that $f$ \eqref{eq:saddlezero} is convex but generally
nonsmooth.  In \cite{SMSN_Nesterov}, Nesterov proposed substituting
$f$ by the smooth approximation
\begin{equation}
\label{eq:smooth-approx}
f_{\mu}(x) = \max_{u \in \mathcal{Q}_d}  \langle u, W x \rangle - \mu\, p_d(u),
\end{equation}
where $p_d(u)$ is a \textit{prox-function} for $\mathcal{Q}_d$; that
is, $p_d(u)$ is continuous and strongly convex on $\mathcal{Q}_d$, with
convexity parameter $\sigma_d$ (we shall assume that $p_d$ vanishes at
some point in $\mathcal{Q}_d$). Nesterov proved that $f_{\mu}$ is
continuously differentiable, and that its gradient obeys 
\begin{equation}\label{eq:dfmu}
\nabla f_{\mu}(x) = W^* u_\mu(x),
\end{equation}
where $u_\mu(x)$ is the optimal solution of
\eqref{eq:smooth-approx}. Furthermore, $\nabla f_{\mu}$ is shown to be
Lipschitz with constant
\begin{equation}
L_{\mu} = \frac{1}{\mu \sigma_d} \|W\|^2 
\end{equation}
($\|W\|$ is the operator norm of $W$). Nesterov's algorithm can then
be applied to $f_{\mu}(x)$ as proposed in \cite{SMSN_Nesterov}. For a
fixed $\mu$, the algorithm converges in $\order(1/k^2)$ iterations.
If we describe convergence in terms of the number of iterations needed
to reach an $\varepsilon$ solution (that is, the number of steps is
taken to produce an $x$ obeying $|f_\mu(x) - \min f_\mu | <
\varepsilon$), then because $\mu$ is approximately proportional to the accuracy of
the approximation, and because $L_\mu$ is proportional to $1/\mu \approx 1/\varepsilon$, the
rate of convergence is $\order( \sqrt{L_\mu/\varepsilon} ) \approx
\order( 1/\varepsilon )$, a significant improvement
over the sub-gradient method which has rate $\order( 1/\varepsilon^2)$.


\section{Extension to Compressed Sensing}\label{sec:nesta}

We now extend Nesterov's algorithm to solve compressed sensing
recovery problems, and refer to this extension as \nestansp.  For now,
we shall be concerned with solving the quadratically constrained
$\ell_1$ minimization problem \eqref{eq:bp}.

\subsection{\nestansp}

We wish to solve \eqref{eq:bp}, i.e. minimize $\|x\|_{\ell_1}$ subject
to $\|b - Ax\|_{\ell_2} \le \epsilon$, where $A \in \mathbb{R}^{m \times
n}$ is singular ($m < n$).

In this section, we assume that $A^*A$ is an orthogonal projector,
i.e.~the rows of $A$ are orthonormal. This is often the case in
compressed sensing applications where it is common to take $A$ as a
submatrix of a unitary transformation which admits a fast algorithm for
matrix-vector products; special instances include the discrete Fourier
transform, the discrete cosine transform, the Hadamard transform, the
noiselet transform, and so on.  Basically,
collecting incomplete structured orthogonal measurements is the prime method for
efficient data acquisition in compressed sensing.

Recall that the $\ell_1$ norm is of the form 
\[
\|x\|_{\ell_1} = \max_{u \in \mathcal{Q}_d} \langle u,x \rangle,  
\]
where the dual feasible set is the $\ell_\infty$ ball
\[
\mathcal{Q}_d = \{ u : \|u\|_{\infty} \leq 1\}. 
\]
Therefore, a natural smooth approximation to the $\ell_1$ norm is
\[
f_\mu(x) =  \max_{u \in \mathcal{Q}_d} \langle u,x \rangle - \mu\, p_d(u), 
\]
where $p_d(u)$ is our dual prox-function. For $p_d(u)$, we would like
a strongly convex function, which is known analytically and takes its
minimum value (equal to zero) at some $u_{d}^{c} \in
\mathcal{Q}_d$. It is also usual to have $p_d(u)$ separable.  Taking
these criteria into account, a convenient choice is $p_d(u) =
\frac{1}{2} \|u\|_{\ell_2}^2$ whose strong convexity parameter
$\sigma_d$ is equal to $1$. With this prox-function, $f_\mu$ is the
well-known Huber function and $\nabla f_\mu$ is Lipschitz with
constant $1/\mu$.\footnote{In the case of total-variation minimization
 in which $f(x) = \|x\|_{TV}$, $f_\mu$ is not a known function.} In
particular, $\nabla f_\mu(x)$ is given by
\begin{equation}
\nabla f_{\mu}(x)[i] = \begin{cases}
\mu^{-1} \, x[i], &  \text{if }   |x[i]| < \mu, \\
\text{sgn}(x[i]), &  \text{otherwise}.
\end{cases}
\label{eq:nabfmu}
\end{equation}
Following Nesterov, we need to solve the smooth constrained problem
\begin{equation}
\min_{x \in \mathcal{Q}_p} f_{\mu}(x), 
\label{eq:l1smooth}
\end{equation}
where $\mathcal{Q}_p = \left \{ x : \|b - Ax \|_{\ell_2} \le \epsilon
\right \}$. Once the gradient of $f_{\mu}$ at $x_k$ is computed, Step
$2$ and Step $3$ of \nesta consist in updating two auxiliary iterates,
namely, $y_k$ and $z_k$.

\subsection{Updating $y_k$}
\label{sec:yk}
To compute $y_k$, we need to solve 
\begin{equation}
y_k = \Argmin{x \in \mathcal{Q}_p} \frac{L_{\mu}}{2} \|x_k - x \|_{\ell_2}^2 + \langle \nabla f_{\mu}(x_k),x-x_k \rangle, 
\label{eq:yk}
\end{equation}
where $x_k$ is given.  The Lagrangian for this problem is of course 
\begin{equation}
\mathcal{L}(x,\lambda) =  \frac{L_{\mu}}{2} \|x_k - x \|_{\ell_2}^2 + \frac{\lambda}{2} \left( \|b - Ax\|_{\ell_2}^2 - \epsilon^2 \right)+ \langle \nabla f_{\mu}(x_k),x-x_k \rangle,
\end{equation}
and at the primal-dual solution $(y_k,\lambda_\epsilon)$, the
Karush-Kuhn-Tucker (KKT) conditions \cite{Rockafellar:1970px} read
\begin{align*}
\|b - Ay_k \|_{\ell_2}^2  & \leq  \epsilon, \\
\lambda_\epsilon  & \geq  0, \\
\lambda_\epsilon \left( \|b - Ay_k \|_{\ell_2}^2 - \epsilon^2 \right) & =  0, \\
L_\mu (y_k - x_k) + \lambda_\epsilon A^*(Ay_k - b) + \nabla f_{\mu}(x_k)& =  0.  
\end{align*}
From the stationarity condition, $y_k$ is the solution to the linear
system
\begin{equation}
\left(I + \frac{\lambda}{L_{\mu}}A^*A \right)y_k = \frac{\lambda}{L_{\mu}} A^*b + x_k - \frac{1}{L_{\mu}} \nabla f_{\mu}(x_k).
\end{equation}
As discussed earlier, our assumption is that $A^*A$ is an orthogonal
projector so that 
\begin{equation}\label{eq:upyk}
y_k = \left(I - \frac{\lambda}{\lambda + L_{\mu}}A^*A \right)
\left(\frac{\lambda}{L_{\mu}} A^*b + x_k - \frac{1}{L_{\mu}} 
\nabla f_{\mu}(x_k)\right).
\end{equation}
In this case, computing $y_k$ is cheap since no matrix inversion is
required---only a few matrix-vector products are necessary.  Moreover,
from the KKT conditions, the value of the optimal Lagrange multiplier
is obtained explicitly, and equals
\begin{equation}
\lambda_\epsilon = \max(0, \epsilon^{-1} \|b - Aq\|_{\ell_2} - L_\mu), \quad q = x_k - L_{\mu}^{-1} \nabla f_{\mu}(x_k).
\label{eq:lambdaKKT}
\end{equation}
Observe that this can be computed beforehand since it only depends on
$x_k$ and $\nabla f_\mu(x_k)$.

\subsection{Updating $z_k$}
To compute $z_k$, we need to solve 
\begin{equation}
z_k = \Argmin{x \in \mathcal{Q}_p} \frac{L_{\mu}}{\sigma_p} p_p(x) + \langle \sum_{i \le k} \alpha_i \nabla f_{\mu}(x_i),x-x_k \rangle, 
\end{equation}
where $p_p(x)$ is the primal prox-function. The point $z_k$ differs
from $y_k$ since it is computed from a weighted cumulative gradient
$\sum_{i \le k} \alpha_i \nabla f_{\mu}(x_i)$, making it less prone to
zig-zagging, which typically occurs when we have highly elliptical level
sets. This step keeps a memory from the previous steps and forces
$z_k$ to stay near the prox-center.

A good primal prox-function is a smooth and strongly convex function
that is likely to have some positive effect near the solution. In the
setting of \eqref{eq:cs}, a suitable smoothing prox-function may be
\begin{equation}\label{eq:primalpx}
p_p(x) =  \frac{1}{2}\|x - x_0 \|_{\ell_2}^2
\end{equation}
for some $x_0 \in \mathbb{R}^n$, e.g.~an initial guess of the
solution. Other choices of primal feasible set $\mathcal{Q}_p$ may
lead to other choices of prox-functions. For instance, when
$\mathcal{Q}_p$ is the standard simplex, choosing an entropy distance
for $p_p(x)$ is smarter and more efficient, see
\cite{SMSN_Nesterov}. In this paper, the primal feasible set is
quadratic, which makes the Euclidean distance a reasonable
choice. What is more important, however, is that this choice allows
very efficient computations of $y_k$ and $z_k$ while other choices may
considerably slow down each Nesterov iteration.  Finally, notice that
the bound on the error at iteration $k$ in \eqref{eq:conv} is
proportional to $p_p(x^\star)$; choosing $x_0$ wisely (a good first
guess) can make $p_p(x^\star)$ small. When nothing is known about the
solution, a natural choice may be $x_0 = A^* b$; this idea will be
developed in Section~\ref{sec:conti}.

With \eqref{eq:primalpx}, the strong convexity parameter of $p_p(x)$
is equal to $1$, and to compute $z_k$ we need to solve
\begin{equation}
z_k = \Argmin{x} \frac{L_{\mu}}{2} \|x - x_0\|_{\ell_2}^2 + \frac{\lambda}{2}\|b - Ax \|_{\ell_2}^2+ \langle \sum_{i \le k} \alpha_i \nabla f_{\mu}(x_i),x-x_k \rangle
\end{equation}
for some value of $\lambda$.  Just as before, the solution is given by
\begin{equation}\label{eq:upzk}
z_k = \left(I - \frac{\lambda}{\lambda + L_{\mu}}A^*A \right)\left(\frac{\lambda}{L_{\mu}} A^*b + x_0 - \frac{1}{L_{\mu}} \sum_{i \leq k}\alpha_i\nabla f_{\mu}(x_i)\right),
\end{equation}
with a value of the Lagrange multiplier equal to
\begin{equation}
\lambda_\epsilon = \max(0, \epsilon^{-1}  \|b - Aq\|_{\ell_2} - L_\mu), \quad q = x_0 - L^{-1}_{\mu}\sum_{i \leq k} \nabla \alpha_i f_{\mu}(x_i).
\end{equation}
In practice, the instances $\{ \nabla f_{\mu}(x_i)\}_{i \leq k}$ have
not to be stored; one just has to store the cumulative gradient
$\sum_{i \leq k} \alpha_i\nabla f_{\mu}(x_i)$.


\subsection{Computational complexity}
The computational complexity of each of \nestansp's step is clear. In
large-scale problems, most of the work is in the application of $A$
and $A^*$. Put $\mathcal{C}_A$ for the complexity of applying $A$ or
$A^*$. The first step, namely, computing $\nabla f_{\mu}$, only
requires vector operations whose complexity is $\mathcal{O}(n)$. Step
$2$ and $3$ require the application of $A$ or $A^*$ three times each
(we only need to compute $A^*b$ once). Hence, the total complexity of
a single \nesta iteration is $6 \, \mathcal{C}_A + \mathcal{O}(n)$
where $\mathcal{C}_A$ is dominant. 

The calculation above are in some sense overly pessimistic.  In
compressed sensing applications, it is common to choose $A$ as a
submatrix of a unitary transformation $U$, which admits a fast
algorithm for matrix-vector products. In the sequel, it might be useful
to think of $A$ as a subsampled DFT. In this case, letting $R$ be the
$m \times n$ matrix extracting the observed measurements, we have $A =
R U$. The trick then is to compute in the $U$-domain directly. Making
the change of variables $x \leftarrow Ux$, our problem is
\[ 
\begin{array}{ll}
\text{minimize}   & \quad \hat{f}_\mu(x)\\
\text{subject to} & \quad  \|b - R x\|_{\ell_2} \le \epsilon,
\end{array}
\]
where $\hat{f}_\mu = f_\mu \circ U^*$. The gradient of
$\hat{f}_\mu$ is then
\[
\nabla \hat{f}_\mu (x) = U \, \nabla f_\mu(U^*x).
\]
With this change of variables, Steps $2$ and $3$ do not require
applying $U$ or $U^*$ since
\[
{y}_k = \left(I - \frac{\lambda}{\lambda + L_{\mu}} R^*R
\right)\left(\frac{\lambda}{L_{\mu}} R^*b + x_k - \frac{1}{L_\mu} \nabla
f_{\mu}(x_k)\right),
\]
where $R^* R$ is the diagonal matrix with $0/1$ diagonal entries
depending on whether a coordinate is sampled or not. As before,
$\lambda_\epsilon = \max(0,\|b-Rq\|_{\ell_2} - L_\mu)$ with $q = {x}_k
- L_{\mu}^{-1} \nabla \hat{f}_{\mu}({x}_k)$.  The complexity of Step
$2$ is now $\order(n)$ and the same applies to Step $3$.

Put $\mathcal{C}_U$ for the complexity of applying $U$ and $U^*$. The
complexity of Step $1$ is now $2 \, \mathcal{C}_{U}$, so that this
simple change of variables reduces the cost of each \nesta iteration
to $2\,\mathcal{C}_U + \mathcal{O}(n)$. For example, in the case of a
subsampled DFT (or something similar), the cost of each iteration is
essentially that of two FFTs. Hence, each iteration is extremely fast.


\subsection{Parameter selection}
\label{sec:paramset}

\nesta involves the selection of a single smoothing parameter $\mu$
and of a suitable stopping criterion. For the latter, our experience
indicates that a robust and fairly natural stopping criterion is to
terminate the algorithm when the relative variation of $f_{\mu}$ is
small. Define $\Delta f_{\mu}$ as
\begin{equation}\label{eq:stopc}
 \Delta f_\mu := \frac{|f_\mu(x_k) - \bar{f}_\mu(x_k)|}{\bar{f}_\mu(x_k)}, 
 \quad \bar{f}_\mu(x_k) := 
 \frac{1}{\min \{10,k\}}\sum_{l=1}^{\min \{10,k\}} f_{\mu}(x_{k-l}).  
\end{equation}
Then convergence is claimed when
\[
\Delta f_{\mu} < \delta
\]
for some $\delta > 0$.  In our experiments, $\delta \in \{10^{-5},
10^{-6}, 10^{-7}, 10^{-8}\}$ depending upon the desired accuracy.

The choice of $\mu$ is based on a trade-off between the accuracy of
the smoothed approximation $f_\mu$ (basically, $\lim_{\mu \rightarrow
 0} f_\mu(x) = \|x\|_{\ell_1}$) and the speed of convergence (the
convergence rate is proportional to $\mu$). With noiseless data, $\mu$
is directly linked to the desired accuracy.  To illustrate this, we
have observed in \cite{BC_Icip09} that when the true signal $x^0$ is
exactly sparse and is actually the minimum solution under the equality
constraints $Ax^0 = b$, the $\ell_\infty$ error on the nonzero entries
is on the order of $\mu$.  The link between $\mu$ and accuracy will be
further discussed in Section~\ref{sec:mu}.




\subsection{Accelerating \nesta with continuation}\label{sec:conti}

Inspired by homotopy techniques which find the solution to the lasso
problem \eqref{eq:lasso} for values of $\tau$ ranging in an interval
$[0,\tau_{\max}]$, \cite{FPC} introduces a fixed point continuation
technique which solves $\ell_1$-penalized least-square problems
\eqref{eq:pls}
\[
(\text{QP}_{\lambda}) \quad \text{minimize} \quad \lambda
\|x\|_{\ell_1} + \frac{1}{2} \|b - Ax\|_{\ell_2}^2,
\]
for values of $\lambda$ obeying $0 < \lambda <
\|A^*b\|_{\ell_\infty}$. The continuation solution approximately
follows the path of solutions to the problem \QPL and, hence, the
solutions to \eqref{eq:cs} and \eqref{eq:lasso} may be found by
solving a sequence a $\ell_1$ penalized least-squares problems.

The point of this is that it has been noticed (see
\cite{FPC,Osborne2000,donoho:l1greedy}) that solving \eqref{eq:pls}
(resp.~the lasso \eqref{eq:lasso}) is faster when $\lambda$ is large
(resp.~$\tau$ is low). This observation greatly motivates the use of
continuation for solving \eqref{eq:pls} for a fixed $\lambda_f$.  The
idea is simple: propose a sequence of problems with decreasing values
of the parameter $\lambda$, $\lambda_0 > \cdots > \lambda_f$, and use
the intermediate solution as a warm start for the next problem. This
technique has been used with some success in
\cite{ist:fnw,SPGL}. Continuation has been shown to be a very
successful tool to increase the
speed of convergence, in particular when dealing with large-scale
problems and high dynamic range signals.

Likewise, our proposed algorithm can greatly benefit from a
continuation approach. Recall that to compute $y_k$, we need to solve
\begin{align*}
y_k & = \Argmin{x \in \mathcal{Q}_p} \frac{L_\mu}{2}\|x - x_k \|_{\ell_2}^2 + \langle c,x \rangle\\
& = \Argmin{x \in \mathcal{Q}_p} \|x - (x_k - L_\mu^{-1} c)\|_{\ell_2}^2
\end{align*}
for some vector $c$.  Thus with $\mathcal{P}_{\mathcal{Q}_p}$ the
projector onto $\mathcal{Q}_p$, $y_k = \mathcal{P}_{\mathcal{Q}_p}(x_k
- L_\mu^{-1} c)$. Now two observations are in order.
\begin{remunerate}
\item Computing $y_k$ is similar to a projected gradient step as the
Lipschitz constant $L^{-1}_{\mu}$ plays the role of the step
size. Since $L_\mu$ is proportional to $\mu^{-1}$, the larger $\mu$,
the larger the step-size, and the faster the convergence. This also
applies to the sequence $\{z_k\}$.  

\item For a fixed value of $\mu$, the convergence rate of the
algorithm obeys 
\[
f_{\mu}(y_k) - f_{\mu}(x^\star_\mu) \le \frac{2 L_\mu \| x^\star_{\mu}
- x_0\|_{\ell_2}^2}{k^2},
\]
where $x^\star_\mu$ is the optimal solution to $\min f_\mu$ over
$\mathcal{Q}_p$. On the one hand, the convergence rate is proportional
to $\mu^{-1}$, so a large value of $\mu$ is beneficial. On the other
hand, choosing a good guess $x_0$ close to $x_\mu^\star$ provides a
low value of $p_p(x_{\mu}^\star) = \frac{1}{2}  \|x^\star_{\mu}
- x_0\|_{\ell_2}^2$, also improving the
rate of convergence.  Warm-starting with $x_0$ from a previous solve
not only changes the starting point of the algorithm, but it
beneficially changes $p_p$ as well.
\end{remunerate}

These two observations motivate the following continuation-like
algorithm: 
\begin{center}
\centering
\vspace{0.15in}
\begin{tabular}{|c|} \hline
\begin{minipage}[hbt]{0.95\linewidth}
\vspace{0.1in}

\textsf{\textbf{Initialize} $\mu_0$, $x_0$ and the number of continuation steps $\CtSteps$. For $t \ge 1$,} \\

\textsf{1. Apply Nesterov's algorithm with $\mu = \mu^{(t)}$ and $x_0 = x_{\mu^{(t-1)}}$.} \\

\textsf{2. Decrease the value of $\mu$: $\mu^{(t+1)} = \gamma \mu^{(t)}$ with $\gamma < 1$.} \\

\textsf{\textbf{Stop} when the desired value of $\mu_f$ is reached.}

\vspace{0.1in}
\end{minipage}
\\\hline
\end{tabular}
\vspace{0.15in}
\end{center}
This algorithm iteratively finds the solutions to a succession of
problems with decreasing smoothing parameters $\mu_0 > \cdots > \mu_f
= \gamma^\CtSteps\mu_0$ producing a sequence of---hopefully--- finer
estimates of $x_{\mu_f}^\star$; these intermediate solutions are cheap
to compute and provide a string of convenient first guess for the next
problem. In practice, they are solved with less accuracy, making them
even cheaper to compute.

The value of $\mu_f$ is based on a desired accuracy as explained in
Section~\ref{sec:paramset}. As for an initial value $\mu_0$,
\eqref{eq:nabfmu} makes clear that the smoothing parameter plays a
role similar to a threshold. A first choice may then be $\mu_0 = 0.9
\|A^*b\|_{\ell_\infty}$. 

We illustrate the good behavior of the
continuation-inspired algorithm by applying \nesta with continuation
to solve a sparse reconstruction problem from partial frequency
data. In this series of experiments, we assess the performance of
\nesta while the dynamic range of the signals to be recovered
increases.

The signals $x$ are $s$-sparse signals---that is, have exactly $s$
nonzero components---of size $n=4096$ and $s = m/40$. Put $\Lambda$
for the indices of the nonzero entries of $x$; the amplitude of each
nonzero entry is distributed uniformly on a logarithmic scale with a
fixed dynamic range. Specifically, each nonzero entry is generated as
follows:
\begin{equation}\label{eq:entriesmod}
x[i] = \eta_1[i] 10^{\alpha \eta_2[i]}, 
\end{equation}
where $\eta_1[i] = \pm 1$ with probability $1/2$ (a random sign) and
$\eta_2[i]$ is uniformly distributed in $[0,1]$. The parameter
$\alpha$ quantifies the dynamic range. Unless specified otherwise, a
dynamic range of $d$ dB means that $\alpha = d/20$ (since for large
signals $\alpha$ is approximately the logarithm base 10 of the ratio
between the largest and the lowest magnitudes). For instance, 80 dB
signals are generated according to \eqref{eq:entriesmod} with $\alpha
= 4$.

The measurements $Ax$ consist of $m = n/8$ random discrete cosine
measurements so that $A^*A$ is diagonalized by the DCT. Finally, $b$
is obtained by adding a white Gaussian noise term with standard
deviation $\sigma = 0.1$. The initial value of the smoothing parameter
is $\mu_0 = \|A^*b\|_{\ell_\infty}$ and the terminal value is $\mu_f =
2\sigma$. The algorithm terminates when the relative variation of
$f_\mu$ is lower than $\delta = 10^{-5}$.  \nesta with continuation is
applied to 10 random trials for varying number of continuation steps
$\CtSteps$ and various values of the dynamic
range. Figure~\ref{fig:continuation1} graphs the value of $f_{\mu_f}$
while applying \nesta with and without continuation as a function of
the iteration count.  The number of continuation steps is set to
$\CtSteps=4$.

\begin{figure}
\begin{center}
\includegraphics[scale=0.37]{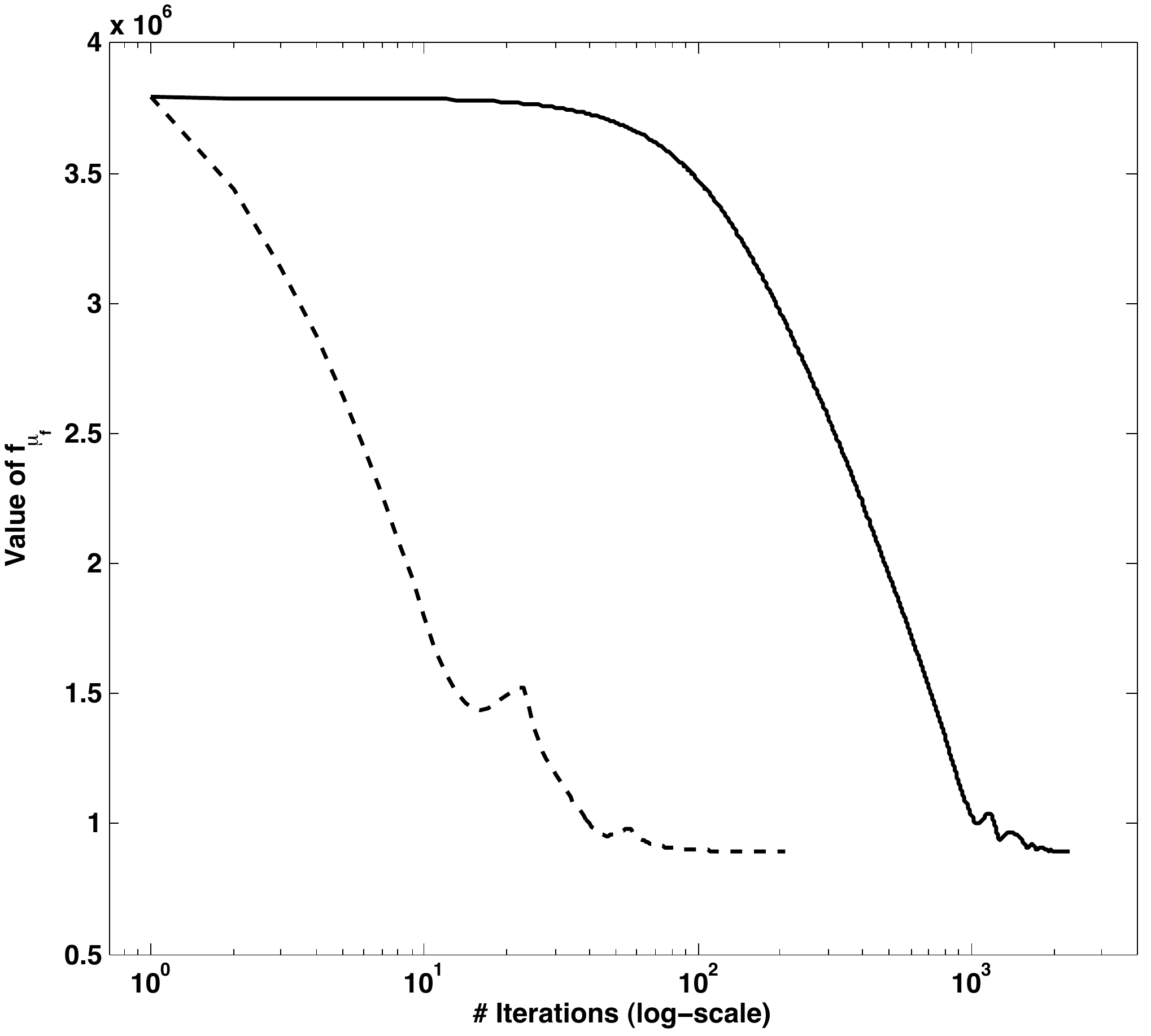}
\caption{Value of $f_{\mu_f}(x_k)$ as a function of iteration
$k$. \textit{Solid line:} without continuation. \textit{Dashed
  line:} with continuation. Here, the test signal has 100 dB of
dynamic range.}
\label{fig:continuation1}
\end{center}
\end{figure}

One can observe that computing the solution to $\min f_{\mu_f}$ (solid
line) takes a while when computed with the final value $\mu_f$; notice
that \nesta seems to be slow at the beginning (number of iterations
lower than 15). In the meantime \nesta with continuation rapidly
estimates a sequence of coarse intermediate solutions that converges
to the solution to $\min f_{\mu_f}$ In this case, continuation clearly
enhances the global speed of convergence with a factor $10$.
Figure~\ref{fig:continuation2} provides deeper insights into the
behavior of continuation with \nesta and shows the number of
iterations required to reach convergence for varying values of the
continuation steps $\CtSteps$ for different values of the dynamic
range.

\begin{figure}
\begin{center}
\includegraphics[scale=0.4]{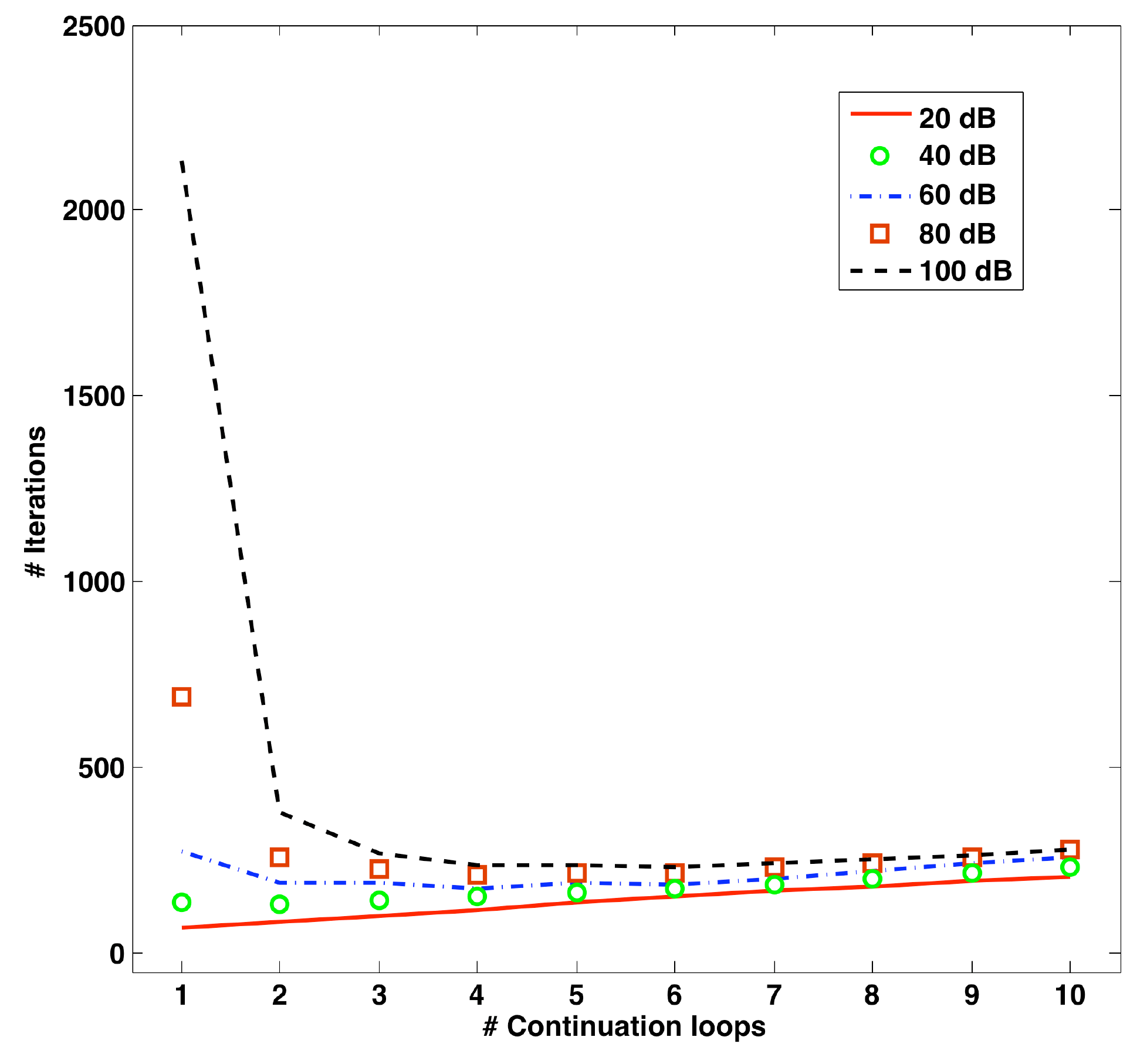}
\caption{Total number of iterations required for convergence with a
varying number of continuation steps and for different values of the
dynamic range.}
\label{fig:continuation2}
\end{center}
\end{figure}

When the ratio $\mu_0/\mu_f$ is low or when the required accuracy is
low, continuation is not as beneficial: intermediate continuation
steps require a number of iterations which may not speed up overall
convergence.  The stepsize which is about $L^{-1}_{\mu_f}$ works well
in this regime. When the dynamic range increases and we require more
accuracy, however, the ratio $\mu_0/\mu_f$ is large, since $\mu_0 = .9
\|A^*b\|_{\ell_\infty} \approx \|x\|_{\ell_\infty} \gg \sigma$, and
continuation provides considerable improvements. In this case, the
step size $L^{-1}_{\mu_{f}}$ is too conservative and it takes a while
to find the large entries of $x$. Empirically, when the dynamic range
is $100$ dB, continuation improves the speed of convergence by a
factor of $8$.  As this factor is likely to increase exponentially
with the dynamic range (when expressed in dB), \nesta with
continuation seems to be a better candidate for solving sparse
reconstruction problems with high accuracy.

Interestingly, the behavior of \nesta with continuation seems to be
quite stable: increasing the number of continuation steps does not
increase dramatically the number of iterations. In practice, although
the ideal $\CtSteps$ is certainly signal dependent, we have observed
that choosing $\CtSteps \in \{4,5,6\}$ leads to reasonable results.


\subsection{Some theoretical considerations}\label{sec:th}

The convergence of \nesta with and without continuation is
straightforward. The following theorem states that each continuation
step with $\mu = \mu^{(t)}$ converges to $x^\star_{\mu}$. Global
convergence is proved by applying this theorem to $t = \CtSteps$.
\begin{theorem}
At each continuation step $t$, $\lim_{k \rightarrow \infty} y_k =
x_{\mu^{(t)}}^\star$, and
\[
f_{\mu^{(t)}}(y_k) - f_{\mu^{(t)}}(x_{\mu^{(t)}}^\star) \le \frac{2
L_{\mu^{(t)}} \|x_{\mu^{(t)}}^\star - x_{\mu^{(t-1)}}
\|_{\ell_2}^2}{k^2}. 
\]
\end{theorem}
\begin{proof}
Immediate by using \cite [Theorem 2]{SMSN_Nesterov}.
\end{proof}

As mentioned earlier, continuation may be valuable for improving the
speed of convergence.  Let each continuation step $t$ stop after
$\mathcal{N}^{(t)}$ iterations with
\[
\mathcal{N}^{(t)} = \sqrt{\frac{2L_{\mu^{(t)}}}{\gamma^t
  \delta_0}}\|x_{\mu^{(t)}}^\star - x_{\mu^{(t-1)}}^\star
\|_{\ell_2}
\]
so that we have 
\[
f_{\mu^{(t)}}(y_k) - f_{\mu^{(t)}}(x_{\mu^{(t)}}^\star) \le
\gamma^t\delta_0,
\] 
where the accuracy $\gamma^t\delta_0$ becomes tighter as $t$
increases.  Then summing up the contribution of all the continuation
steps gives
\[
\mathcal{N}_c = \sqrt{\frac{2}{\mu_0 \delta_0}}
\sum_{t=1}^\CtSteps\gamma^{-t}\|x_{\mu^{(t)}}^\star -
x_{\mu^{(t-1)}}^\star \|_{\ell_2}. 
\]

When \nesta is applied without continuation, the number of iterations
required to reach convergence is
\[
\mathcal{N} = \sqrt{\frac{2}{\mu_0 \delta_0}} \gamma^{-\CtSteps}
\|x_{\mu_f}^\star - x_0 \|_{\ell_2}.
\]
Now the ratio $\mathcal{N}_c/\mathcal{N}$ is given by
\begin{equation}\label{eq:nupbound}
\frac{\mathcal{N}_c}{\mathcal{N}} = 
\sum_{t=1}^\CtSteps\gamma^{\CtSteps-t}\frac{\|x_{\mu^{(t)}}^\star 
- x_{\mu^{(t-1)}}^\star \|_{\ell_2}}{\|x_{\mu_f}^\star - x_0 \|_{\ell_2}}. 
\end{equation}
Continuation is definitely worthwhile when the right-hand side is
smaller than $1$. Interestingly, this quantity is directly linked to
the path followed by the sequence $x_0 \rightarrow
x_{\mu^{(1)}}\rightarrow \cdots \rightarrow x_{\mu_f}$. More
precisely, it is related to the smoothness of this path; for instance,
if all the intermediate points $x_{\mu^{(t)}}$ belong to the segment
$[x_0,x_{\mu_f}]$ in an ordered fashion, then $\sum_t
\|x_{\mu^{(t)}}^\star - x_{\mu^{(t-1)}} \|_{\ell_2} =
\|x_{\mu_f}^\star - x_0 \|_{\ell_2}$. Hence,
$\frac{\mathcal{N}_c}{\mathcal{N}} < 1$ and continuation improves the
convergence rate.

Figure~\ref{fig:solpaths} illustrates two typical solution paths with
continuation. When the sequence of solutions obeys $\|x_0\|_{\ell_1}
\geq \ldots \|x_{\mu^{(t)}}^\star\|_{\ell_1} \ldots \geq
\|x_{\mu_f}^\star\|_{\ell_1}$ (this is the case when $x_0 = A^*b$ and
$\mu_1 \geq \ldots \mu^{(t)} \ldots \geq \mu_f$), the solution path is
likely to be ``smooth;'' that is, the solutions obey
$\|x_{\mu^{(t)}}^\star - x_{\mu_f}^\star\|_{\ell_2} \geq
\|x_{\mu^{(t+1)}}^\star - x_{\mu_f}^\star\|_{\ell_2}$ as on the left
of Figure~\ref{fig:solpaths}. The ``nonsmooth'' case on the right of
Figure~\ref{fig:solpaths} arises when the sequence of smoothing
parameters does not provide estimates of $x_{\mu_f}^\star$ that are
all better than $x_0$. Here, computing some of the intermediate points
$\{x_{\mu^{(t)}}^\star\}$ is wasteful and continuation fails to be
faster.

\begin{center}
\begin{figure}
\includegraphics[scale=0.35]{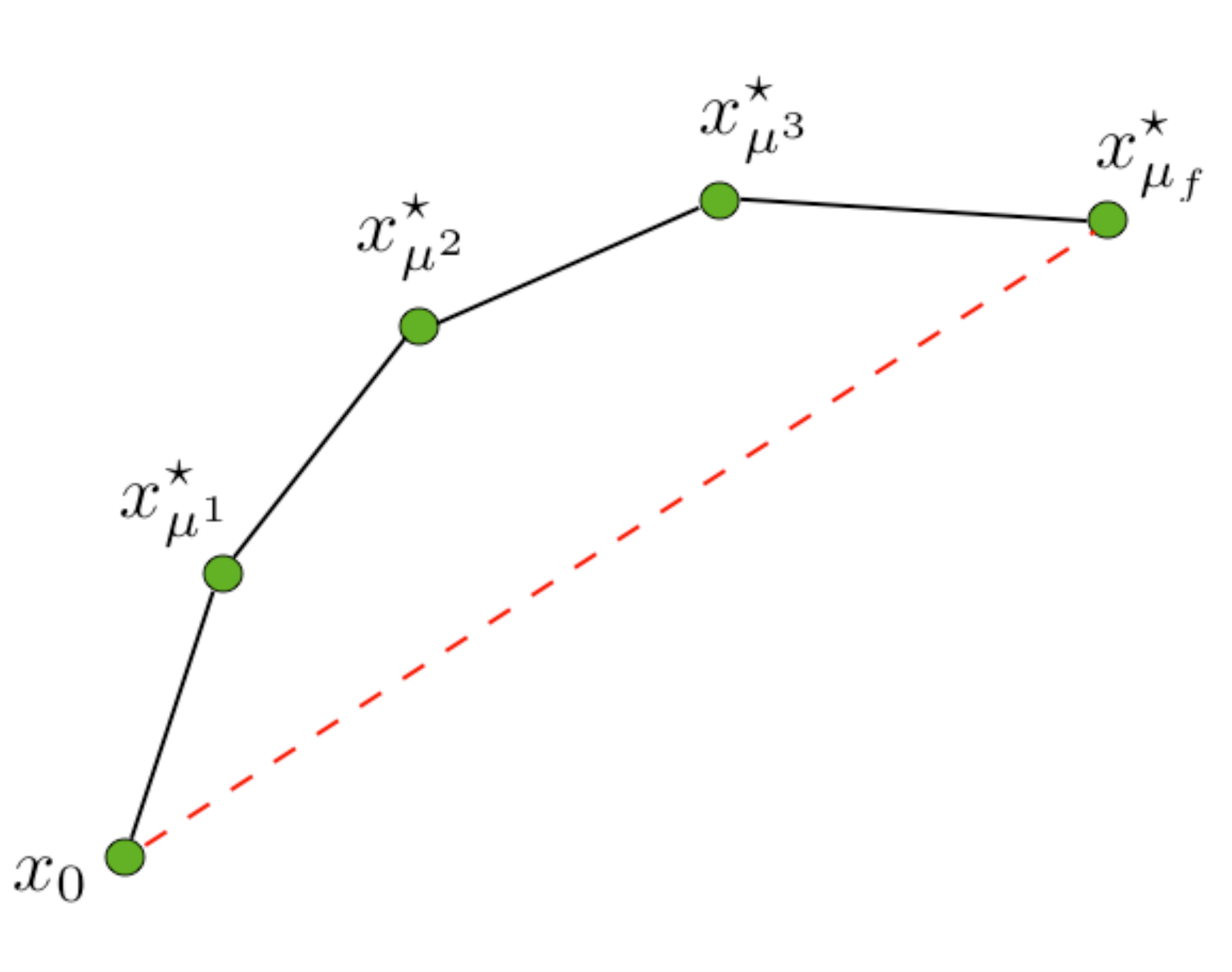}
\hfill
\includegraphics[scale=0.35]{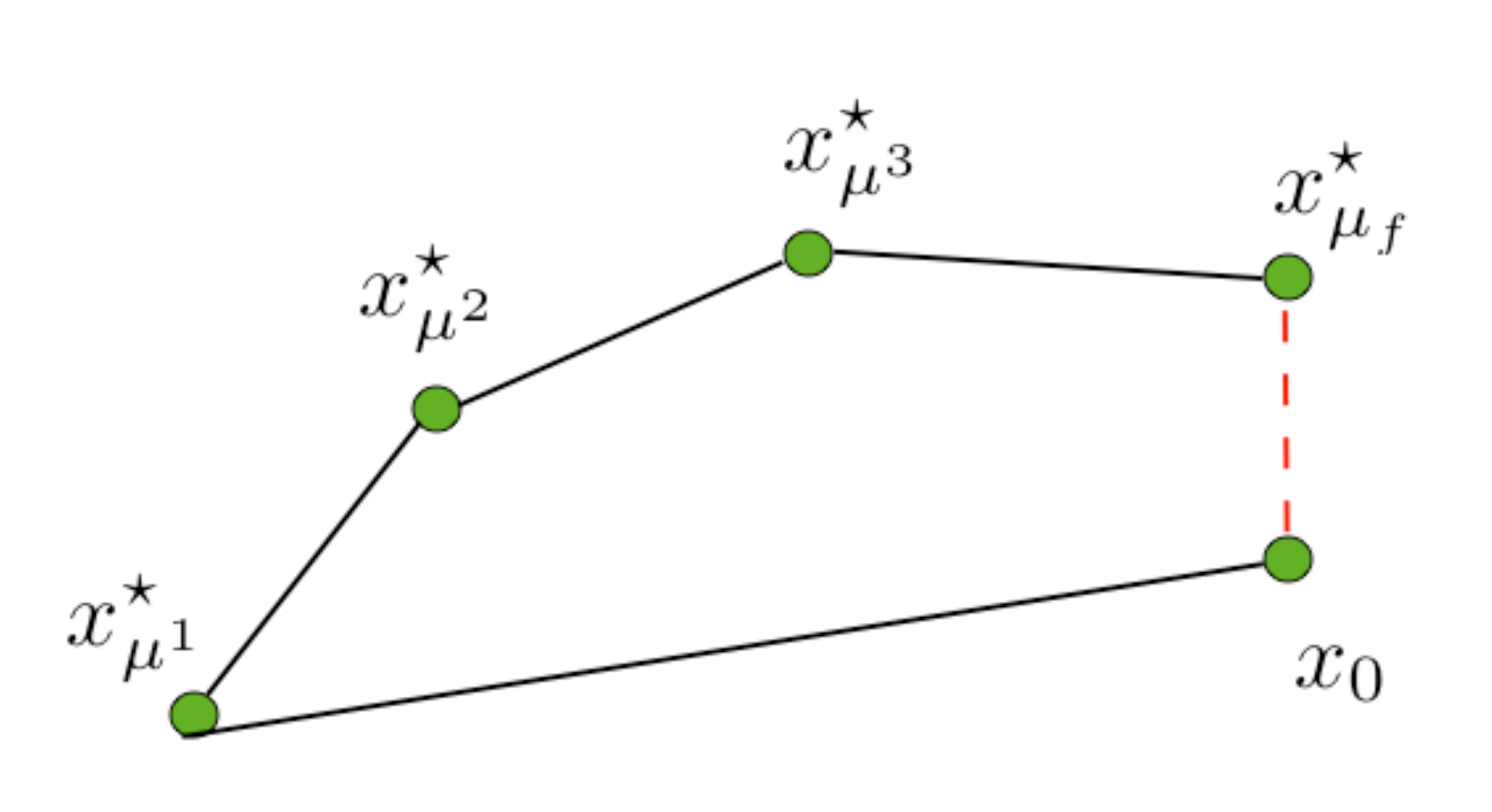}

\caption{Typical solution paths  - \textbf{Left:} smooth solution path. \textbf{Right:} nonsmooth solution path.}
\label{fig:solpaths}
\end{figure}
\end{center}


\section{Accurate Optimization}
\label{sec:new}

A significant fraction of the numerical part of this paper focuses on
comparing different sparse recovery algorithms in terms of speed and
accuracy. In this section, we first demonstrate that \nesta can easily
recover the exact solution to \BPE with a precision of 5 to 6
digits. Speaking of precision, we shall essentially use two criteria
to evaluate accuracy.
\begin{remunerate}
\item The first is the (relative) error on the objective functional
\begin{equation}
\label{eq:relerror}
\frac{\|x\|_{\ell_1} - \|x^\star\|_{\ell_1}}{\|x^\star\|_{\ell_1}},
\end{equation}
where $x^\star$ is the optimal solution to \BPE.
\item The second is the accuracy of the optimal solution itself and is
 measured via 
\begin{equation}
\label{eq:linf}
\|x - x^\star \|_{\ell_\infty},
\end{equation}
which gives a precise value of the accuracy per entry.
\end{remunerate}

\subsection{Is \nesta accurate?\hspace*{-.15cm}}
\label{sec:nestamu}

For general problem instances, the exact solution to \BPE (or
equivalently $(\text{QP}_{\lambda})$) cannot be computed analytically. Under some
conditions, however, a simple formula is available when the optimal
solution has exactly the same support and the same sign as the unknown
(sparse) $x^0$ (recall the model $b = A x^0 +z$). Denote by $I$ the
support of $x^0$, $I := \{i : |x^0[i]| > 0\}$. Then if $x^0$ is
sufficiently sparse and if the nonzero entries of $x^0$ are
sufficiently large, the solution $x^\star$ to \QPL is given by
\begin{align}
\label{eq:optimsol}
x^\star[I] & = (A[I]^* A[I])^{-1}(A[I]^* b - \lambda\, \text{sgn}(x^0[I])),\\
x^\star[{I^c}] & = 0,
\end{align}
see \cite{CP07} for example. In this expression, $x[I]$ is the vector
with indices in $I$ and $A[I]$ is the submatrix with columns indices
in $I$. 

To evaluate \nestansp's accuracy, we set $n = 262,\!144$, $m = n/8$,
and $s = m/100$ (this is the number of nonzero coordinates of $x_0$).
The absolute values of the nonzero entries of $x_0$ are distributed
between $1$ and $10^5$ so that we have about $100$ dB of dynamic
range. 
The measurements $Ax^0$ are discrete cosine coefficients selected
uniformly at random. We add Gaussian white noise with standard
deviation $\sigma = 0.01$. We then compute the solution
\eqref{eq:optimsol}, and make sure it obeys the KKT optimality
conditions for \QPL so that the optimal solution is known.

\begin{table}
\caption{Assessing FISTA's and \nestansp's accuracy when the optimal 
 solution is known. The relative error on the optimal value is given by 
 \eqref{eq:relerror} and the $\ell_\infty$ error on the optimal 
 solution by \eqref{eq:linf}. $\mathcal{N}_A$ is the number of calls 
 to $A$ or $A^*$ to compute the solution.}

\begin{center}\footnotesize
\begin{tabular}{l||c|c|c|c}
  Method & $\ell_1$-norm & Rel. error $\ell_1$-norm & $\ell_\infty$ error &  $\mathcal{N}_A$ \\
  \hline
  $x^\star$ & 3.33601e+6  & &  &   \\
  \hline
  FISTA &  3.33610e+6 & 2.7e-5 & 0.31  & 40000\\
  \nesta $\mu = 0.02$ & 3.33647e+6 & 1.4e-4 & 0.08  & 513 \\
\end{tabular}
\end{center}
\label{Fistaccurate}
\end{table}

We run \nesta with continuation with the value of $\epsilon := \|b -
Ax^\star\|$. We use $\mu = 0.02$, $\delta = 10^{-7}$ and the number of
continuation steps is set to $5$.  Table~\ref{Fistaccurate} reports on
numerical results. First, the value of the objective functional is
accurate up to $4$ digits. Second, the computed solution is very
accurate since we observe an $\ell_\infty$ error of $0.08$. Now recall
that the nonzero components of $x^\star$ vary from about $1$ to $10^5$
so that we have high accuracy over a huge dynamic range.  This can
also be gleaned from Figure~\ref{fig:optvsfista} which plots
\nestansp's solution versus the optimal solution, and confirms the
excellent precision of our algorithm.

\begin{figure}
\begin{centering}
\includegraphics[scale=0.6]{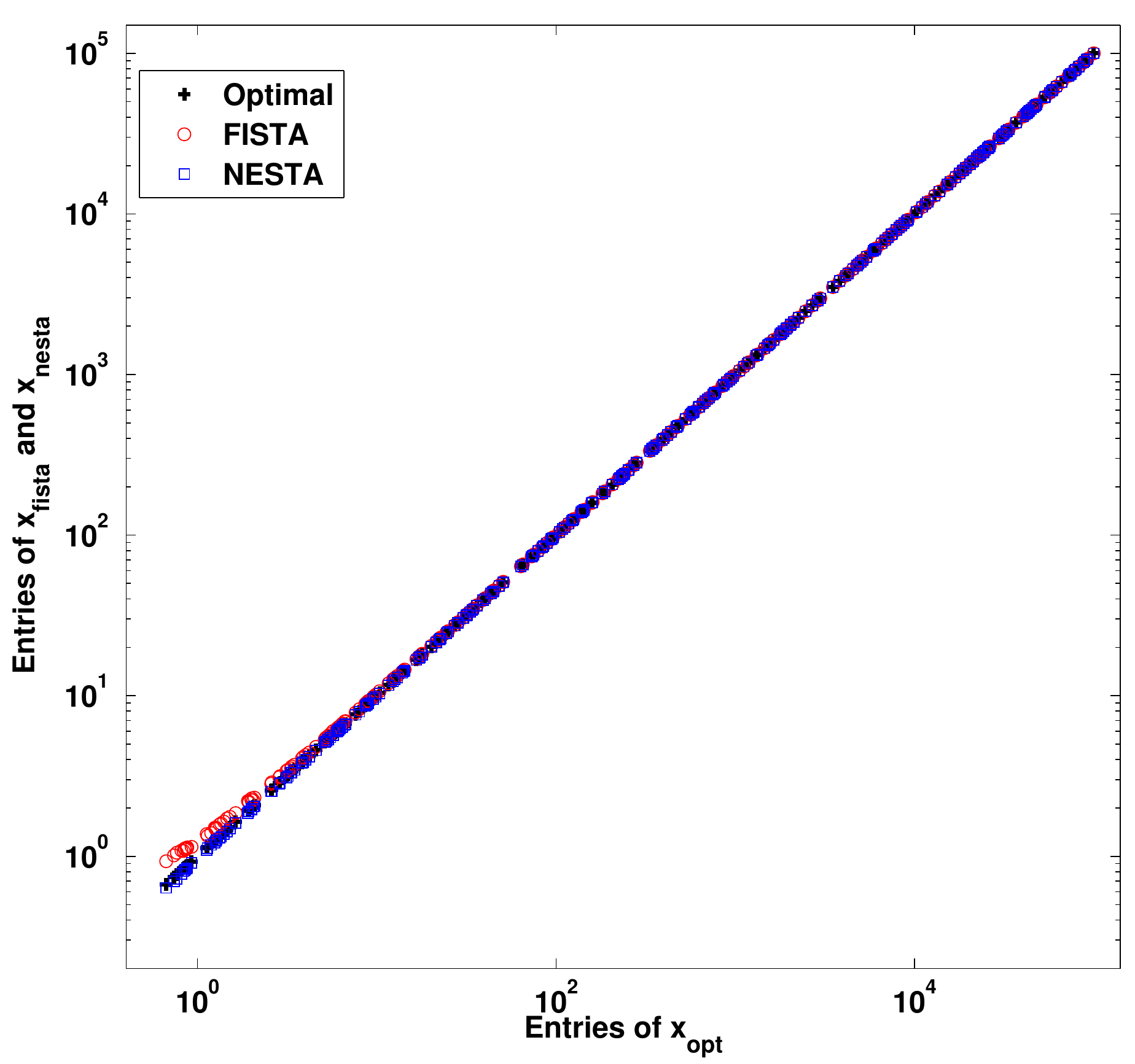}
\caption{Entries of the computed solutions versus the optimal
 solution. The absolute values of the entries on the support of the
 optimal solution are plotted.}
\label{fig:optvsfista}
\end{centering}
\end{figure}

\subsection{Setting up a reference algorithm for accuracy tests}
\label{sec:fistaref}

In general situations, a formula for the optimal solution is of course
unavailable, and evaluating the accuracy of solutions requires
defining a method of reference. In this paper, we will use FISTA
\cite{BeckT08} as such a reference since it is an efficient algorithm
that also turns out to be extremely easy to use; in particular, no
parameter has to be tweaked, except for the standard stopping
criterion (maximum number of iterations and tolerance on the relative
variation of the objective function).

We run FISTA with $20,\!000$ iterations on the same problem as above,
and report its accuracy in Table~\ref{Fistaccurate}. The $\ell_1$-norm
is exact up to $4$ digits. Furthermore, Figure~\ref{fig:optvsfista}
shows the entries of FISTA's solution versus those of the optimal
solution, and one observes a very good fit (near perfect when the
magnitude of a component of $x^\star$ is higher than $3$).  The
$\ell_\infty$ error between FISTA's solution and the optimal solution
$x^\star$ is equal to $0.31$; that is, the entries are exact up to
$\pm 0.31$. Because this occurs over an enormous dynamic range, we
conclude that FISTA also gives very accurate solutions provided that
sufficiently many iterations are taken.  We have observed that running
FISTA with a high number of iterations---typically greater than
$20,\!000$---provides accurate solutions to $(\text{QP}_{\lambda})$, and
this is why we will use it as our method of reference in the
forthcoming comparisons from this section and the next.

\subsection{The smoothing parameter $\mu$ and \nestansp's accuracy}
\label{sec:mu}

By definition, $\mu$ fixes the accuracy of the approximation $f_\mu$
to the $\ell_1$ norm and, therefore, \nestansp's accuracy directly
depends on this parameter. We now propose to assess the accuracy of
\nesta for different values of $\mu$. The problem sizes are as before,
namely, $n = 262,\!144$ and $m = n/8$, except that now the unknown
$x^0$ is far less sparse with $s = m/5$. The standard deviation of the
additive Gaussian white noise is also higher, and we set $\sigma =
0.1$.

Because of the larger value of $s$ and $\sigma$, it is no longer
possible to have an analytic solution from \eqref{eq:optimsol}.
Instead, we use FISTA to compute a reference solution $x_F$, using
$20,\!000$ iterations and with $\lambda = 0.0685$, which gives $\|b -
Ax_F\|^2_{\ell_2} \simeq (m + 2\sqrt{2m}) \sigma^2$.  To be sure that
FISTA's solution is very close to the optimal solution, we check that
the KKT stationarity condition is nearly verified. If $I_\star$ is the
support of the optimal solution $x^\star$, this condition reads
\begin{align*}
 A[I_\star]^*(b - A x^\star) & = \lambda \, \text{sgn}(x^\star[I_\star]),\\
 \|A[I_\star^c]^*(b - A x^\star)\|_{\ell_\infty} & \le \lambda.
\end{align*}
Now define $I$ to be the support of $x_F$. Then, here, $x_F$ obeys
\begin{align*}
 \|A[I]^*(b - A x_F) - \lambda \, \text{sgn}(x_F[I])\|_{\ell_\infty}
 & = 2.66 10^{-10} \lambda,\\
 \|A[I^c]^*(b - A x_F)\|_{\ell_\infty} & \le 0.99 \lambda.
\end{align*}
This shows that $x_F$ is extremely close to the optimal solution.

\begin{figure}
\begin{centering}
\includegraphics[scale=0.6]{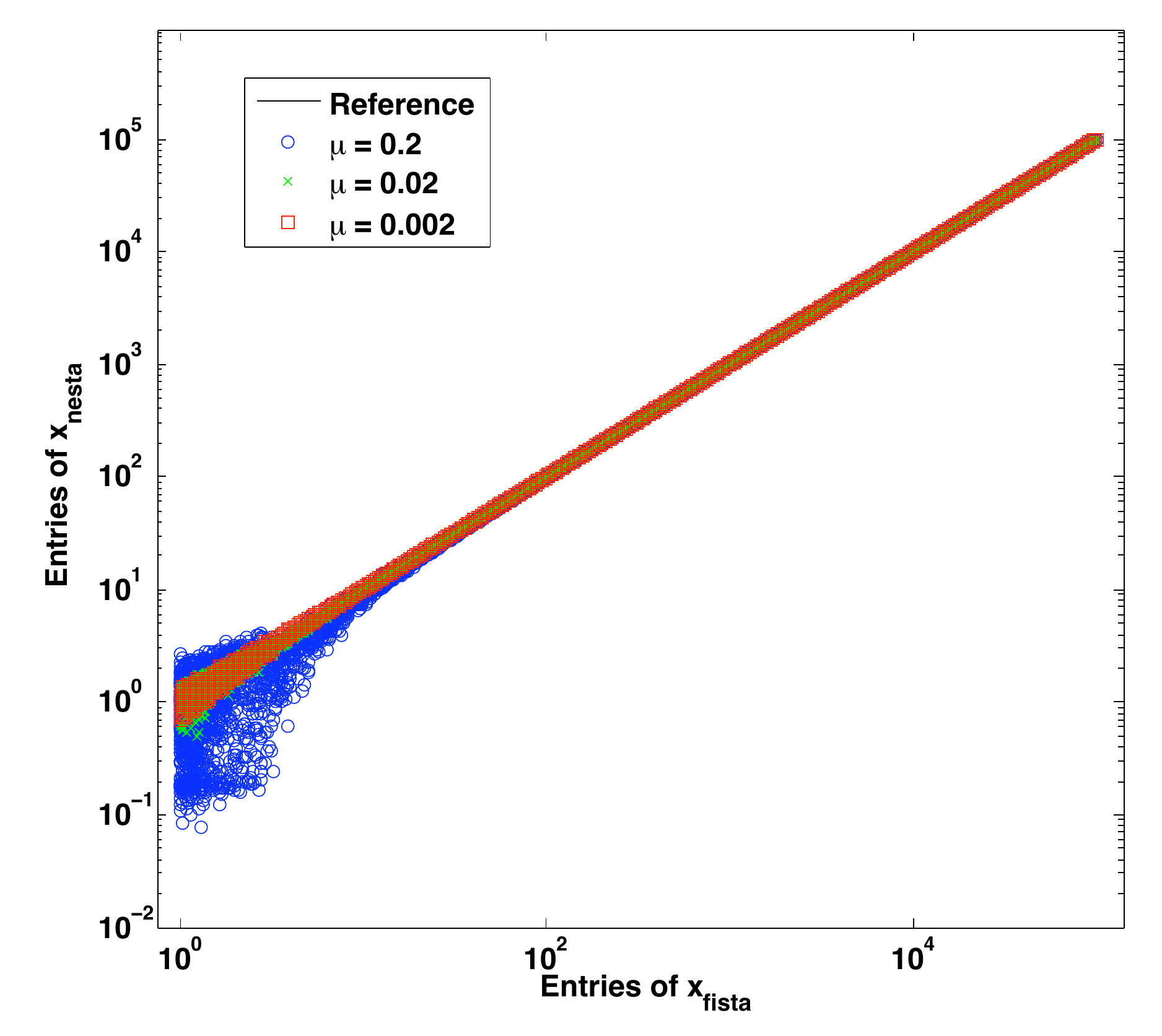}
\caption{Entries of the computed solutions versus the optimal
 solution. We plot the absolute values of the entries on the set
 where the magnitude of the optimal solution exceeds 1.}
\label{fig:nestavsfista}
\end{centering}
\end{figure}

\nesta is run with $T=5$ continuation steps for three different values
of $\mu \in \{0.2, 0.02, 0.002\}$ (the tolerance $\delta$ is set to
$10^{-6}$, $10^{-7}$ and $10^{-8}$ respectively).
Figure~\ref{fig:nestavsfista} plots the solutions given by \nesta
versus the ``optimal solution'' $x_F$. Clearly, when $\mu$ decreases,
the accuracy of \nesta increases just as expected.  More precisely,
notice in Table~\ref{Nestaccurate} that for this particular
experiment, decreasing $\mu$ by a factor of $10$ gives about $1$
additional digit of accuracy on the optimal value. 
\begin{table}
 \caption{\nestansp's accuracy. The errors and number of function calls $\mathcal{N}_A$ have the same meaning as in Table~\ref{Fistaccurate}.} 

\begin{center}\footnotesize
\begin{tabular}{l||c|c|c|c}
 Method & $\ell_1$-norm & Rel. error $\ell_1$-norm & $\ell_\infty$ error & $\mathcal{N}_A$ \\
 \hline
 FISTA & 5.71539e+7 & &  &  \\
 \hline
 \nesta $\mu = 0.2$ &  5.71614e+7 & 1.3e-4 & 3.8 & 659\\
 \nesta $\mu = 0.02$ & 5.71547e+7 & 1.4e-5 & 0.96  & 1055\\
 \nesta $\mu = 0.002$ & 5.71540e+7 & 1.6e-6 & 0.64   & 1537
\end{tabular}
\end{center}
\label{Nestaccurate}
\end{table}

According to this table, $\mu =0.02$ seems a reasonable choice to
guarantee an accurate solution since one has between $4$ and $5$
digits of accuracy on the optimal value, and since the $\ell_\infty$
error is lower than $1$. Observe that this value separates the nonzero
entries from the noise floor (when $\sigma = 0.01$). In the extensive
numerical experiments of Section~\ref{sec:numeric}, we shall set $\mu
= 0.02$ and $\delta = 10^{-7}$ as default values.


\section{Numerical comparisons}
\label{sec:numeric}

This section presents numerical experiments comparing several
state-of-the-art optimization techniques designed to solve
\eqref{eq:bp} or \eqref{eq:pls}.  To be as fair as possible, we
propose comparisons with methods for which software is publicly
available online.  To the best of our knowledge, such extensive
comparisons are currently unavailable. Moreover, whereas publications
sometimes test algorithms on relatively easy and academic problems, we
will subject optimization methods to hard but realistic $\ell_1$
reconstruction problems.

In our view, a challenging problem involves some or all of the
characteristics below. 
\begin{remunerate}
\item {\em High dynamic range.} As mentioned earlier, most
optimization techniques are able to find (more or less rapidly) the
most significant entries (those with a large amplitude) of the
signal $x$. Recovering the entries of $x$ that have low magnitudes
accurately is more challenging.

\item {\em Approximate sparsity.} Realistic signals are seldom exactly
sparse and, therefore, coping with approximately sparse signals is
of paramount importance. In signal or image processing for example,
wavelet coefficients of natural images contain lots of low level
entries that are worth retrieving.

\item {\em Large scale.} Some standard optimization techniques, such
as interior point methods, are known to provide accurate
solutions. However, these techniques are not applicable to
large-scale problems due to the large cost of solving linear
systems. Further, many existing software packages fail to take
advantage of fast-algorithms for applying $A$. We will focus on
large-scale problems in which the number of unknowns $n$ is over
a quarter of a million, i.e.~$n = 262,\!144$.
\end{remunerate}

\subsection{State-of-the-art methods}

Most of the algorithms discussed in this section are considered to be
state-of-art in the sense that they are the most competitive among
sparse reconstruction algorithms.  To repeat ourselves, many of these
methods have been improved after several years of research
\cite{l1ls,ist:fnw}, and many did not exist two years ago
\cite{FPC,SPGL}. For instance, \cite{FPC2} was submitted for
publication less than three months before we put the final touches on
this paper. Finally, our focus is on rapid algorithms so that we are
interested in methods which can take advantage of fast algorithms for
applying $A$ to a vector. This is why we have not tested other good
methods such as \cite{CoordinateDescent}, for example.

\subsubsection{\nestansp}

Below, we applied \nesta with the following default parameters
\[
x_0 = A^*b, \quad \mu = 0.02, \quad \delta = 10^{-7}
\]
(recall that $x_0$ is the initial guess).  The maximal number of
iterations is set to $\mathcal{I}_{\max} = 10,\!000$; if convergence
is not reached after $\mathcal{I}_{\max}$ iterations, we record that
the algorithm did not convergence (DNC). Because \nesta requires 2
calls to either $A$ or $A^*$ per iteration, this is equivalent to
declaring DNC after $\mathcal{N}_A = 20,\!000$ iterations where
$\mathcal{N}_A$ refers to the total number of calls to $A$ or $A^*$;
hence, for the other methods, we declare DNC when $\mathcal{N}_A >
20,\!000$.  When continuation is used, extra parameters are set up as
follows:
\[
\CtSteps  =   4,   \quad 
\mu_0  =  \|x_0\|_{\ell_\infty}, \quad   
\gamma  =  (\mu/\mu_0)^{1/\CtSteps},
\]
and for $t = 1,\ldots, T$, 
\[
\mu_t = \gamma^t \mu_0,\quad
\delta_t  =  0.1\cdot(\delta/0.1)^{t/\CtSteps}.
\]
Numerical results are reported and discussed in
Section~\ref{sec:numreshd}.

\subsubsection{Gradient Projections for Sparse Reconstruction (GPSR)
\cite{ist:fnw}}

GPSR has been introduced in \cite{ist:fnw} to solve the standard
$\ell_1$ minimization problem in Lagrangian form
($\text{QP}_\lambda$). GPSR is based on the well-known projected
gradient step technique,
\[
v^{(k+1)} = \mathcal{P}_{\mathcal{Q}}\left(v^{(k-1)} -\alpha_k \nabla
F(v_k)\right),
\]
for some projector $\mathcal{P}_{\mathcal{Q}}$ onto a convex set
$\mathcal{Q}$; this set contains the variable of interest $v$. In this
equation, $F$ is the function to be minimized. In GPSR, the problem is
recast such that the variable $v = [v_1,v_2]$ has positive entries and
$x = v_1 - v_2$ (a standard change of variables in linear programming
methods). The function $F$ is then
\[
F(v) = \lambda \, \underline{1}^* v +\frac{1}{2}\|b - [A , -A] v
\|_{\ell_2}^2, 
\]
where $\underline{1}$ is the vector of ones, and $v$ belongs to the
nonnegative orthant, $v[i] \geq 0$ for all $i$. The projection onto
$\mathcal{Q}$ is then trivial.  Different techniques for choosing the
step-size $\alpha_k$ (backtracking, Barzilai-Borwein \cite{BB88}, and
so on) are discussed in \cite{ist:fnw}. The code is available at
\url{http://www.lx.it.pt/~mtf/GPSR/}. In the forthcoming experiments,
the parameters are set to their default values.

GPSR also implements continuation, and we test this version as well.
All parameters were set to defaults except, per the recommendation of
one of the GPSR authors to increase performance, the number of
continuation steps was set to 40, the {\tt ToleranceA} variable was
set to $10^{-3}$, and the {\tt MiniterA} variable was set to $1$.  In
addition, the code itself was tweaked a bit; in particular, the
stopping criteria for continuation steps (other than the final step)
was changed.  Future releases of GPSR will probably contain a
similarly updated continuation stopping criteria.

\subsubsection{Sparse reconstruction by separable approximation
  (SpaRSA) \cite{sparsa}}

SpaRSA is an algorithm to minimize composite functions $\phi(x) = f(x)
+ \lambda c(x)$ composed of a smooth term $f$ and a separable
non-smooth term $c$, e.g.~($\text{QP}_\lambda$).  At every step, a
subproblem of the form
\[
\text{minimize} \quad \|x-y\|_{\ell_2}^2 + \frac{\lambda}{\alpha} c(x)
\]
with optimization variable $x$ must be solved; this is the same as
computing the proximity operator corresponding to $c$.  For
($\text{QP}_\lambda$), the solution is given by shrinkage.  In this
sense, SpaRSA is an iterative shrinkage/thresholding (IST) algorithm,
much like FISTA (though without the accelerated convergence) and FPC.
Also like FPC, continuation is used to speed convergence, and like
FPC-BB, a Barzilai-Borwein heuristic is used for the step size
$\alpha$ (instead of using a pessimistic bound like the Lipschitz
constant).  With this choice, SpaRSA is not guaranteed to be monotone,
which can be remedied by implementing an appropriate safeguard,
although this is not done in practice because there is little
experimental advantage to doing so.  Code for SpaRSA may be obtained
at \url{http://www.lx.it.pt/~mtf/SpaRSA/}.  Parameters were set to
default except the number of continuation steps was set to 40 and the
{\tt MiniterA} variable was set to 1 (instead of the default 5), as
per the recommendations of one of the SpaRSA authors---again, as to
increase performance.

\subsubsection{$\ell_1$ regularized least squares (l1\_ls) \cite{l1ls}}

This method solves the standard unconstrained $\ell_1$ minimization
problem, and is an interior point method (with log-barrier) using
preconditioned conjugate gradient (PCG) to accelerate convergence and
stabilize the algorithm. The preconditioner used in the PCG step is a
linear combination of the diagonal approximation of the Hessian of the
quadratic term and of the Hessian of the log-barrier term.  l1\_ls is
shown to be faster than usual interior point methods; nevertheless,
each step requires solving a linear system of the form $H \Delta x =
g$. Even if PCG makes the method more reliable, l1\_ls is still
problematic for large-scale problems. In the next comparisons, we
provide some typical values of its computational complexity compared
to the other methods. The code is available at \url{http://www.stanford.edu/~boyd/l1_ls/}.

\subsubsection{Spectral projected gradient (SPGL1) \cite{SPGL}}

In 2008, van den Berg et al.~adapted the spectral projection gradient
algorithm introduced in \cite{Birg00} to solve the LASSO
($\text{LS}_\tau$). Interestingly, they introduced a clever root finding
procedure such that solving a few instances of \LST for different
values of $\tau$ enables them to equivalently solve
($\text{BP}_\epsilon$).  Furthermore, if the algorithm
detects a nearly-sparse solution, it defines an active set and solves
an equation like \eqref{eq:optimsol} on this active set.
In the next experiments, the parameters are
set to their default values. The code is available at
\url{http://www.cs.ubc.ca/labs/scl/SPGL11/}.

\subsubsection{Fixed Point Continuation method (FPC) \cite{FPC,FPC2}}

The Fixed Point Continuation method is a recent first-order algorithm
for solving \QPL and simple generalizations of \QPL\!.  The main idea
is based on a fixed point equation, $ x = F(x) $, which holds at the
solution (derived from the subgradient optimality condition).  For
appropriate parameters, $F$ is a contraction, and thus the algorithm $
x_{k+1} = F( x_k )$ converges.  The operator $F$ comes from
forward-backward splitting, and consists of a
soft-thresholding/shrinkage step and a gradient step.  The main
computational burden is one application of $A$ and $A^*$ at every
step.  The papers \cite{FPC,FPC2} prove $q$-linear convergence, and
finite-convergence of some of the components of $x$ for $s$-sparse
signals.  The parameter $\lambda$ in \QPL determines the amount of
shrinkage and, therefore, the speed of convergence; thus in practice,
$\lambda$ is decreased in a continuation scheme.  Code for FPC is
available at
\url{http://www.caam.rice.edu/~optimization/L1/fpc/}. Also available
is a state-of-the-art version of FPC from 2008 that uses
Barzilai-Borwein \cite{BB88} steps to accelerate performance.  In the
numerical tests, the Barzilai-Borwein version (referred to as FPC-BB)
significantly outperforms standard FPC.  All parameters were set to
default values.

\subsubsection{FPC Active Set (FPC-AS) \cite{FPCAS}}

In 2009, inspired by both first-order algorithms, such as FPC, and
greedy algorithms \cite{donoho:stgomp,CoSaMP}, Wen et al.~\cite{FPCAS}
extend FPC into the two-part algorithm FPC Active Set to solve
\QPL\!\!.  In the first stage, FPC-AS calls an improved version of FPC
that allows the step-size to be updated dynamically, using a
non-monotone exact line search to ensure $r$-linear convergence, and a
Barzilai-Borwein \cite{BB88} heuristic.  After a given stopping
criterion, the current value, $x_k$, is hard-thresholded to determine
an active set.  On the active set, $ \|x\|_{\ell_1}$ is replaced by
$c^* x$, where $c = \text{sgn}(x_k)$, with the constraints that
$x[i]\cdot c[i] > 0$ for all the indices $i$ belonging to the active
set.  The objective is now smooth, and solvers, like conjugate
gradients (CG) or quasi-Newton methods (e.g.~L-BFGS or L-BFGS-B), can
solve for $x$ on the active set; this the same as solving
\eqref{eq:optimsol}.  This two-step process is then repeated for a
smaller value of $\lambda$ in a continuation scheme.  We tested FPC-AS
using both L-BFGS (the default) and CG (which we refer to as
FPC-AS-CG) to solve the subproblem; both of these solvers do not
actually enforce the $x[i]\cdot c[i] > 0$ constraint on the active
set.  Code for FPC-AS is available at
\url{http://www.caam.rice.edu/~optimization/L1/FPC_AS/}.

For $s$-sparse signals, all parameters were set to defaults except for the
stopping criteria (as discussed in Section \ref{sec:protocol}).  For approximately sparse signals, FPC-AS
performed poorly ($ > 10,000$ iterations) with the default
parameters.  By changing a parameter that controls the {\em estimated}
number of nonzeros from $m/2$ (default) to $n$, the performance
improved dramatically, and this is the performance reported in the
tables.  The maximum number of subspace iterations was also changed from 
the default to 10, as recommended in the help file.

\subsubsection{Bregman}

The Bregman Iterative algorithm, motivated by the Bregman distance, has
been shown to be surprisingly simple \cite{YinCS08}.  The first
iteration solves \QPL for a specified value of $\lambda$; subsequent
iterations solve \QPL for the same value of $\lambda$, with an updated
observation vector $b$.  Typically, only a few outer iterations are
needed (e.g.~4), but each iteration requires a solve of
($\text{QP}_\lambda$), which is costly.  The original Bregman
algorithm calls FPC to solve these subproblems; we test Bregman using
FPC and the Barzilai-Borwein version of FPC as subproblem solvers.

A version of the Bregman algorithm, known as the Linearized Bregman
algorithm \cite{linBreg,COS:XXX:08}, takes only one step of the inner
iteration per outer iteration; consequently, many outer iterations are
taken, in contrast to the regular Bregman algorithm. It can be shown
that linearized Bregman is equivalent to gradient ascent on the dual
problem.  Linearized Bregman was not included in the tests because no
standardized public code is available.  Code for the regular Bregman
algorithm may be obtained at
\url{http://www.caam.rice.edu/~optimization/L1/2006/10/bregman-iterative-algorithms-for.html}.
There are quite a few parameters, since there are parameters for the outer
iterations and for the inner (FPC) iterations; for all experiments,
parameters were set to defaults.  In particular, we noted that using the default
stopping criteria for the inner solve, which limited FPC to $1,\!000$ iterations, led to significantly
better results than allowing the subproblem to run to $10,\!000$ iterations.

\subsubsection{Fast Iterative Soft-Thresholding Algorithm (FISTA)} 
\label{sec:fistadescription}

FISTA is based upon Nesterov's work but departs from \nesta in two
important ways: 1) FISTA solves the sparse unconstrained
reconstruction problem ($\text{QP}_\lambda$); 2) FISTA is a proximal
subgradient algorithm, which only uses two sequences of iterates. In
some sense, FISTA is a simplified version of the algorithm previously
introduced by Nesterov to minimize composite functions
\cite{NestCompo}. The theoretical rate of convergence of FISTA is similar to
\nestansp's, and has been shown to decay as $\mathcal{O}(1/k^2)$.

For each test, FISTA is run twice: it is first run until the relative
variation in the function value is less than $10^{-14}$, with no limit on
function calls, and this solution is used as the reference solution.
It is then run a second time using either Criteria 1 or Criteria 2
as the stopping condition, and these are the results reported
in the tables.

\subsection{Constrained versus unconstrained minimization}
\label{sec:constrained}

We would like to briefly highlight the fact that these algorithms are
not solving the same problem.  \nesta and SPGL1 solve the constrained
problem ($\text{BP}_{\epsilon}$), while all other methods tested solve
the unconstrained problem ($\text{QP}_\lambda$).  As the first chapter
of any optimization book will emphasize, solving an unconstrained
problem is in general much easier than a constrained
problem.\footnote{The constrained problem ($\text{BP}_{\epsilon}$) is
  equivalent to that of minimizing $\|x\|_{\ell_1} +
  \chi_{\mathcal{Q}_p}(x)$ where $\mathcal{Q}_p$ is the feasible set
  $\{x : \|Ax - b \|_{\ell_2} \le \epsilon\}$, and
  $\chi_{\mathcal{Q}_p}(x) = 0$ if $x \in \mathcal{Q}_p$ and $+\infty$
  otherwise. Hence, the unconstrained problem has a discontinuous
  objective functional.} For example, it may be hard to even find a
feasible point for ($\text{BP}_\epsilon$), since the pseudo-inverse of
$A$, when $A$ is not a projection, may be difficult to compute.  It is
possible to solve a sequence of unconstrained
($\text{QP}_{\lambda_j}$) problems for various $\lambda_j$ to
approximately find a value of the dual variable $\lambda$ that leads
to equivalence with ($\text{BP}_\epsilon$), but even if this procedure
is integrated with the continuation procedure, it will require
several, if not dozens, of solves of \QPL (and this will in general
only lead to an approximate solution to ($\text{BP}_{\epsilon}$)).
The Newton-based root finding method of SPGL1 relies on solving a
sequence of {\em constrained} problems ($\text{LS}_\tau$); basically,
the dual solution to a constrained problem gives useful information.

Thus, we emphasize that SPGL1 and \nesta are actually more general
than the other algorithms (and as Section~\ref{sec:flexible} shows,
\nesta is even more general because it handles a wide variety of
constrained problems); this is especially important because from a
practical viewpoint, it may be easier to estimate an appropriate
$\epsilon$ than an appropriate value of $\lambda$.  Furthermore, as
will be shown in Section~\ref{sec:numreshd}, SPGL1 and \nesta with
continuation are also the most robust methods for arbitrary signals
(i.e.~they perform well even when the signal is not exactly sparse,
and even when it has high dynamic range).  Combining these two facts,
we feel that these two algorithms are extremely useful for real-world
applications.

\subsection{Experimental protocol}\label{sec:protocol}

In these experiments, we compare \nesta with other efficient methods.
There are two main difficulties with comparisons which might explain
why broad comparisons have not been offered before.  The first problem
is that some algorithms, such as \nestansp, solve
($\text{BP}_{\epsilon}$), whereas other algorithms solve
($\text{QP}_\lambda$).  Given $\epsilon$, it is difficult to compute
$\lambda(\epsilon)$ that gives an equivalence between the problems; in
theory, the KKT conditions give $\lambda$, but we have observed in
practice that because we have an approximate solution (albeit a very
accurate one), computing $\lambda$ in this fashion is not stable.

Instead, we note that given $\lambda$ and a solution $x_\lambda$ to
($\text{QP}_\lambda$), it is easy to compute a very accurate
$\epsilon(\lambda)$ since $\epsilon = \| A x_\lambda - b \|_{\ell_2}$.
Hence, we use a two-step procedure.  In the first step, we choose a
value of $\epsilon_0= \sqrt{m + 2\sqrt{2m}} \sigma$ based on the noise
level $\sigma$ (since a value of $\lambda$ that corresponds to
$\sigma$ is less clear), and use SPGL1 to solve
($\text{BP}_\epsilon$).  From the SPGL1 dual solution, we have an
estimate of $\lambda = \lambda(\epsilon_0)$.  As noted above, this
equivalence may not be very accurate, so the second step is to compute
$\epsilon_1 = \epsilon(\lambda)$ via FISTA, using a very high accuracy
of $\delta=10^{-14}$.  The pair $(\lambda,\epsilon_1)$ now leads to
nearly equivalent solutions of ($\text{QP}_\lambda$) and
($\text{BP}_\epsilon$).  The solution from FISTA will also
be used to judge the accuracy of the other algorithms.

The other main difficulty in comparisons is a fair stopping criterion.
Each algorithm has its own stopping criterion (or may offer a choice
of stopping criteria), and these are not directly comparable.  To
overcome this difficulty, we have modified the codes of the algorithms
to allow for two new stopping criterion that we feel are the only fair choices.
The short story is that we use \nesta to compute a solution $x_N$ and
then ask the other algorithms to compute a solution that is at least
as accurate.

Specifically, given \nestansp's solution $x_N$ (using continuation),
the other algorithms terminate at iteration $k$ when the solution
$\hat{x}_k$ satisfies
\begin{equation}
\label{stopcond}
\text{(Crit.~1)} \quad 
\|\hat{x}_k\|_{\ell_1} \le  \|x_N\|_{\ell_1} \quad \text{and} \quad 
\|b - A\hat{x}_k\|_{\ell_2} \le 1.05 \, \|b - Ax_N\|_{\ell_2},
\end{equation}
or
\begin{equation}
\label{stopcond3}
 \text{(Crit.~2)}
 \quad 
 \lambda \|\hat{x}_{k}\|_{\ell_1}  + \frac{1}{2} \|A \hat{x}_k -b\|_{\ell_2}^2
 \leq \lambda \|x_N\|_{\ell_1}  + \frac{1}{2} \|Ax_N -b\|_{\ell_2}^2. 
\end{equation}
We run tests with both stopping criteria to reduce any potential bias
from the fact that some algorithms solve ($\text{QP}_\lambda$), for
which Crit.~2 is the most natural, while others solve
($\text{BP}_\epsilon$), for which Crit.~1 is the most natural.  In
practice, the results when applying Crit.~1 or Crit.~2 are not
significantly different.

%
%
%

\subsection{Numerical results}
\label{sec:numreshd}

\subsubsection{The case of exactly sparse signals}

This first series of experiments tests all the algorithms discussed
above in the case where the unknown signal is $s$-sparse with $s =
m/5$, $m = n/8$, and $n=262,\!144$.  This situation is close to the
limit of perfect recovery from noiseless data.  The $s$ nonzero
entries of the signals $x^0$ are generated as described in
\eqref{eq:entriesmod}.  Reconstruction is performed with several
values of the dynamic range $d = 20, 40, 60, 80, 100$ in dB.  The
measurement operator is a randomly subsampled discrete cosine
transform, as in Section~\ref{sec:nestamu} (with a different random
set of measurements chosen for each trial).  The noise level is set to
$\sigma = 0.1$.  The results are reported in Tables \ref{table2}
(Crit.~1) and \ref{table1} (Crit.~2); each cell in these table
contains the mean value of $\mathcal{N}_{A}$ (the number of calls of
$A$ or $A^*$) over $10$ random trials, and, in smaller font, the
minimum and maximum value of $\mathcal{N}_{A}$ over the $10$ trials.
When convergence is not reached after $\mathcal{N}_A = 20,\!000$, we
report $\mbox{DNC}$ (did not converge).  As expected, the number of
calls needed to reach convergence varies a lot from an algorithm to
another.

\begin{table}
\centering
 \caption{Number of function calls $\mathcal{N}_A$ averaged over 10 independent
   runs. The sparsity level $s = m/5$ and the stopping rule is Crit.~1 \eqref{stopcond}.}

\footnotesize
\begin{tabular*}{0.99\textwidth} {@{\extracolsep{\fill}}l||c|c|c|c|c}
  Method & 20 dB & 40 dB & 60 dB & 80 dB & 100 dB \\
   \hline
       \nesta     & 446 \mmax{351}{491}   & 880 \mmax{719}{951}   & 1701 \mmax{1581}{1777}        & 4528 \mmax{4031}{4749}        & 14647 \mmax{7729}{15991}\\ 
       \nesta + Ct     & 479 \mmax{475}{485}   & 551 \mmax{539}{559}   & 605 \mmax{589}{619}   & 658 \mmax{635}{679}   & 685 \mmax{657}{705}\\ 
       GPSR     & 56 \mmax{44}{62}      & 733 \mmax{680}{788}   & 5320 \mmax{4818}{5628}        & \text{DNC}    & \text{DNC}\\ 
       GPSR + Ct	& 305 \mmax{293}{311}	& 251 \mmax{245}{257}	& 497 \mmax{453}{531}	& 1816 \mmax{1303}{2069}	& 9101 \mmax{7221}{10761}\\ 
       SpaRSA	& 345 \mmax{327}{373}	& 455 \mmax{435}{469}	& 542 \mmax{511}{579}	& 601 \mmax{563}{629}	& 708 \mmax{667}{819}\\
       SPGL1     & 54 \mmax{37}{61}      & 128 \mmax{102}{142}   & 209 \mmax{190}{216}   & 354 \mmax{297}{561}   & 465 \mmax{380}{562}\\ 
       FISTA     & 68 \mmax{66}{69}      & 270 \mmax{261}{279}   & 935 \mmax{885}{969}   & 3410 \mmax{2961}{3594}        & 13164 \mmax{11961}{13911}\\ 
       FPC AS     & 156 \mmax{111}{177}   & 236 \mmax{157}{263}   & 218 \mmax{215}{239}   & 351 \mmax{247}{457}   & 325 \mmax{313}{335}\\
       FPC AS (CG)     & 312 \mmax{212}{359}   & 475 \mmax{301}{538}   & 434 \mmax{423}{481}   & 641 \mmax{470}{812}   & 583 \mmax{567}{595}\\ 
       FPC     & 414 \mmax{394}{436}   & 417 \mmax{408}{422}   & 571 \mmax{546}{594}   & 945 \mmax{852}{1038}  & 3945 \mmax{2018}{4734}\\ 
       FPC-BB     & 148 \mmax{140}{152}   & 166 \mmax{158}{168}   & 219 \mmax{208}{250}   & 264 \mmax{252}{282}   & 520 \mmax{320}{800}\\
       Bregman-BB     & 211 \mmax{203}{225}   & 270 \mmax{257}{295}   & 364 \mmax{355}{393}   & 470 \mmax{429}{501}   & 572 \mmax{521}{657}
\end{tabular*}
\label{table2}
\end{table}

\begin{table}
\centering
 \caption{Number of function calls $\mathcal{N}_A$ averaged over 10 independent
   runs. The sparsity level $s = m/5$ and the stopping rule is Crit.~2 \eqref{stopcond3}.}

\footnotesize
\begin{tabular*}{0.99\textwidth} {@{\extracolsep{\fill}}l||c|c|c|c|c}
  Method & 20 dB & 40 dB & 60 dB & 80 dB & 100 dB \\
   \hline
   \nesta     & 446 \mmax{351}{491}   & 880 \mmax{719}{951}   & 1701 \mmax{1581}{1777}        & 4528 \mmax{4031}{4749}        & 14647 \mmax{7729}{15991}\\ 
   \nesta + Ct     & 479 \mmax{475}{485}   & 551 \mmax{539}{559}   & 605 \mmax{589}{619}   & 658 \mmax{635}{679}   & 685 \mmax{657}{705}\\ 
   GPSR     & 59 \mmax{44}{64}      & 736 \mmax{678}{790}   & 5316 \mmax{4814}{5630}        & \text{DNC}     & \text{DNC}\\ 
   GPSR + Ct	& 305 \mmax{293}{311}	& 251 \mmax{245}{257}	& 511 \mmax{467}{543}	& 1837 \mmax{1323}{2091}	& 9127 \mmax{7251}{10789}\\ 
   SpaRSA	& 345 \mmax{327}{373}	& 455 \mmax{435}{469}	& 541 \mmax{509}{579}	& 600 \mmax{561}{629}	& 706 \mmax{667}{819}\\
   SPGL1     & 55 \mmax{37}{61}      & 138 \mmax{113}{152}   & 217 \mmax{196}{233}   & 358 \mmax{300}{576}   & 470 \mmax{383}{568}\\ 
   FISTA     & 65 \mmax{63}{66}      & 288 \mmax{279}{297}   & 932 \mmax{882}{966}   & 3407 \mmax{2961}{3591}        & 13160 \mmax{11955}{13908}\\ 
   FPC AS     & 176 \mmax{169}{183}   & 236 \mmax{157}{263}   & 218 \mmax{215}{239}   & 344 \mmax{247}{459}   & 330 \mmax{319}{339}\\ 
   FPC AS (CG)     & 357 \mmax{343}{371}   & 475 \mmax{301}{538}   & 434 \mmax{423}{481}   & 622 \mmax{435}{814}   & 588 \mmax{573}{599}\\ 
   FPC     & 416 \mmax{398}{438}   & 435 \mmax{418}{446}   & 577 \mmax{558}{600}   & 899 \mmax{788}{962}   & 3866 \mmax{1938}{4648}\\ 
   FPC-BB     & 149 \mmax{140}{154}   & 172 \mmax{164}{174}   & 217 \mmax{208}{254}   & 262 \mmax{248}{286}   & 512 \mmax{308}{790}\\ 
   Bregman-BB     & 211 \mmax{203}{225}   & 270 \mmax{257}{295}   & 364 \mmax{355}{393}   & 470 \mmax{429}{501}   & 572 \mmax{521}{657}
\end{tabular*}
\label{table1}
\end{table}

The careful reader will notice that Tables~\ref{table2} and
\ref{table1} do not feature the results provided by l1\_ls; indeed,
while it seems faster than other interior point methods, it is still
far from being comparable to the other algorithms reviewed here. In
these experiments l1\_ls typically needed $1500$ calls to $A$ or $A^*$
for reconstructing a $20$ dB signal with $s= m/100$ nonzero
entries. For solving the same problem with a dynamic range of $100$
dB, it took $~5$ hours to converge on a dual core MacPro G5 clocked at
2.7GHz.

GPSR performs well in the case of low-dynamic range signals; its
performance, however, decreases dramatically as the dynamic range
increases; Table \ref{table1} shows that it does not converge for 80
and 100 dB signals.  GPSR with continuation does worse on the low
dynamic range signals (which is not surprising).  It does much better
than the regular GPSR version on the high dynamic range signals,
though it is slower than \nesta with continuation by more than a
factor of 10.  SpaRSA performs well at low dynamic range, comparable
to \nestansp, and begins to outperform GSPR with continuation as the
dynamic range increases, although it begins to underperform \nesta
with continuation in this regime. SpaRSA takes over twice as many
function calls on the 100 dB signal as on the 20 dB signal.

SPGL1 shows good performance with very sparse signals and low dynamic
range.  Although it has fewer iteration counts than \nestansp, the
performance decreases much more quickly than for \nesta as the
dynamic range increases; SPGL1 requires about 9$\times$ more calls to
$A$ at 100 dB than at 20 dB, whereas \nesta with continuation requires
only about 1.5$\times$ more calls.  FISTA is almost as fast as SPGL1
on the low dynamic range signal, but degrades very quickly as the
dynamic range increases, taking about 200$\times$ more iterations at
100 dB than at 20 dB.  One large contributing factor to this poor
performance at high dynamic range is the lack of a continuation
scheme.

FPC performs well at low dynamic range, but is very slow on 100 dB
signals.  The Barzilai-Borwein version was consistently faster than
the regular version, but also degrades much faster than \nesta with
continuation as the dynamic range increases. Both FPC Active Set and
the Bregman algorithm perform well at all dynamic ranges, but again,
degrade faster than \nesta with continuation as the dynamic range
increases.  There is a slight difference between the two FPC Active
set versions (using L-BFGS or CG), but the dependence on the dynamic
range is roughly similar.

The performances of \nesta with continuation are reasonable when the
dynamic range is low. When the dynamic range increases, continuation
helps by dividing the number of calls up to a factor about $20$, as in
the 100 dB case. In these experiments, the tolerance $\delta$ is
consistently equal to $10^{-7}$; while this choice is reasonable when
the dynamic range is high, it seems too conservative in the low
dynamic range case. Setting a lower value of $\delta$ should improve
\nestansp's performance in this regime. In other words, \nesta with
continuation might be tweaked to run faster on the low dynamic range
signals. However, this is not in the spirit of this paper and this is
why we have not researched further refinements.

In summary, for exactly sparse signals exhibiting a significant
dynamic range, 1) the performance of \nesta with continuation---but
otherwise applied out-of-the-box---is comparable to that of
state-of-the-art algorithms, and 2) most state-of-the-art algorithms
are efficient on these types of signals.

\subsubsection{Approximately sparse signals}
\label{sec:numresap}

We now turn our attention to approximately sparse signals. Such
signals are generated via a permutation of the Haar wavelet
coefficients of a $512 \times 512$ natural image. The data $b$ are $m
= n/8 = 32,\!768$ discrete cosine measurements selected at random.
White Gaussian noise with standard deviation $\sigma = 0.1$ is then
added.  Each test is repeated 5 times, using a different random
permutation every time (as well as a new instance of the noise
vector).  Unlike in the exactly sparse case, the wavelet coefficients
of natural images mostly contain mid-range and low level coefficients
(see Figure~\ref{fig:wtcoeffs}) which are challenging to recover.

\begin{figure}
\begin{centering}
\includegraphics[scale=0.6]{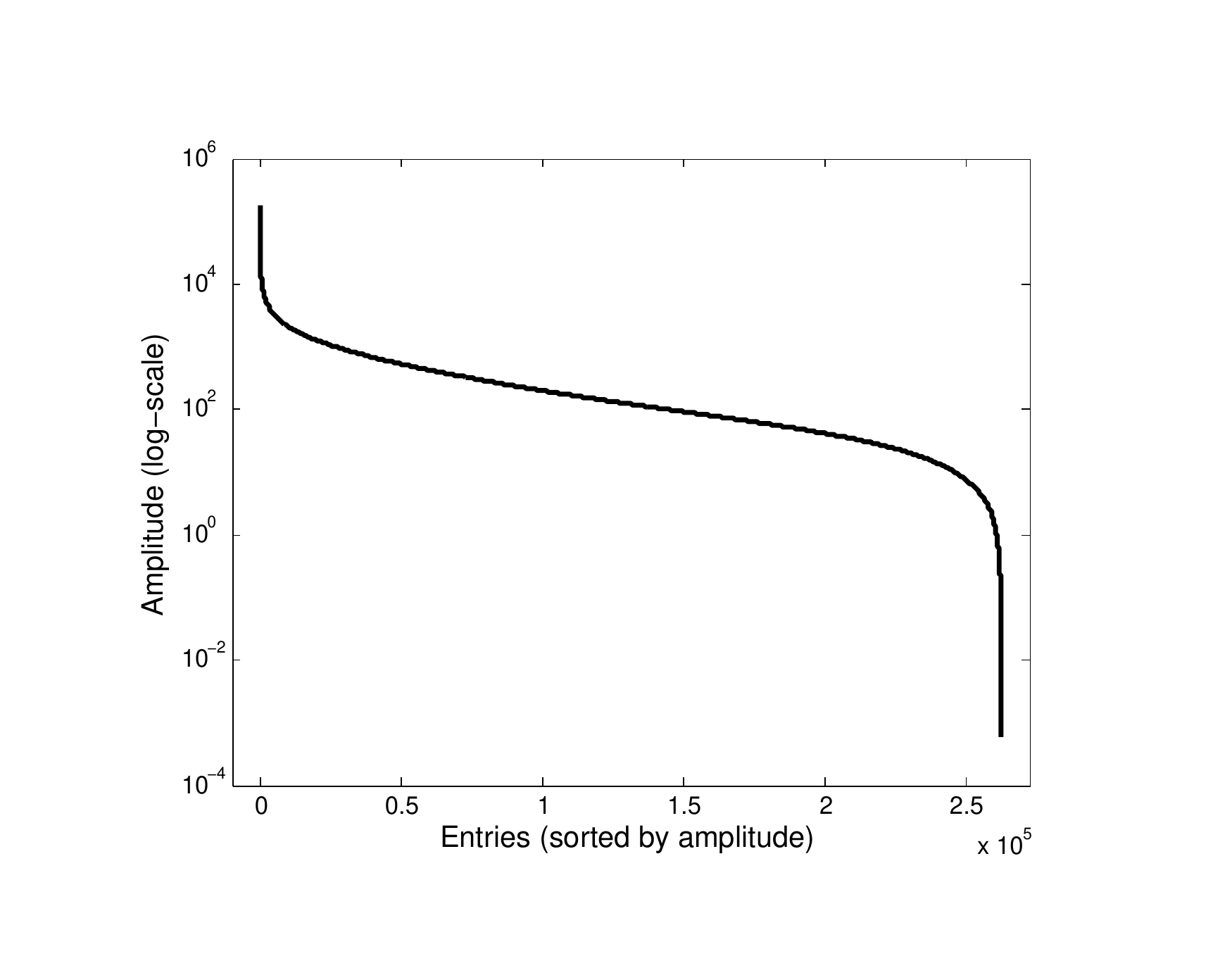}
\caption{Sorted wavelet coefficients of the natural image used in the experiments.}
\label{fig:wtcoeffs}
\end{centering}
\end{figure}

The results are reported in Tables~\ref{table12} (Crit.~1) and
\ref{table14} (Crit.~2); the results from applying the two stopping
criteria are nearly identical.  In these series of experiments, the
performance of SPGL1 is quite good but seems to vary a lot from one
trial to another (Table~\ref{table14}). Notice that the concept of an
active-set is ill defined in the approximately sparse case; as a
consequence, the active set version of FPC is not much of an
improvement over the regular FPC version.  FPC is very fast for
$s$-sparse signals but lacks the robustness to deal with less ideal
situations in which the unknown is only approximately sparse.

FISTA and SpaRSA converge for these tests, but are not competitive
with the best methods.  It is reasonable to assume that FISTA would
also improve if implemented with continuation.  SpaRSA already uses
continuation but does not match its excellent performance on exactly
sparse signals. 

Bregman, SPGL1, and \nesta with continuation all have excellent
performances (continuation really helps NESTA) in this series of
experiments.  \nesta with continuation seems very robust when high
accuracy is required.  The main distinguishing feature of \nesta is
that it is less sensitive to dynamic range; this means that as the
dynamic range increases, or as the noise level $\sigma$ decreases,
\nesta becomes very competitive.  For example, when the same test was
repeated with more noise ($\sigma = 1$), all the algorithms converged
faster.  In moving from $\sigma = 1$ to $\sigma = 0.1$, SPGL1 required
$90\%$ more iterations and Bregman required $20\%$ more iterations,
while \nesta with continuation required only $5\%$ more iterations.

One conclusion from these tests is that SPGL1, Bregman and \nesta
(with continuation) are the only methods dealing with approximately
sparse signals effectively.  The other methods, most of which did very
well on exactly sparse signals, take over $10,\!000$ function calls or
even do not converge in $20,\!000$ function calls; by comparison,
SPGL1, Bregman and \nesta with continuation converge in about
$2,\!000$ function calls.  It is also worth noting that Bregman is
only as good as the subproblem solver; though not reported here, using
the regular FPC (instead of FPC-BB) with Bregman leads to much worse
performance.

The algorithms which did converge all achieved a mean relative
$\ell_1$ error (using \eqref{eq:relerror} and the high accuracy FISTA
solution as the reference) less than $2\cdot 10^{-4}$ and sometimes as
low as $10^{-5}$, except SPGL1, which had a mean relative error of
$1.1\cdot 10^{-3}$.  Of the algorithms that did not converge in
$20,\!000$ function calls, FPC and FPC-BB had a mean $\ell_1$ relative
error about $5\cdot 10^{-3}$, GPSR with continuation had errors about
$5\cdot 10^{-2}$, and the rest had errors greater than $10^{-1}$.

\begin{table}
 \caption{Recovery results of an approximately sparse signal with Crit.~$1$ as a stopping rule.}

\footnotesize
\begin{center} 
 \begin{tabular*}{0.7\textwidth} {@{\extracolsep{\fill}}l||c|c|c}
  Method & $ <\mathcal{N}_A> $ & $\min \mathcal{N}_A$  & $\max \mathcal{N}_A$\\
   \hline
\nesta     & 18912 & 18773 & 19115\\
      \nesta + Ct     & 2667  & 2603  & 2713\\
      GPSR     & DNC & DNC & DNC\\
      GPSR + Ct	& DNC & DNC & DNC\\
      SpaRSA &  10019 &       8369 &      12409\\
      SPGL1     & 1776  & 1073  & 2464\\
      FISTA     & 10765 & 10239 & 11019\\
      FPC Active Set     & DNC & DNC & DNC\\
      FPC Active Set (CG)     & DNC & DNC & DNC\\
      FPC     & DNC & DNC & DNC\\
      FPC-BB     & DNC &DNC & DNC\\
      Bregman-BB     & 2045  & 2045  & 2045
\end{tabular*}
\label{table12}
\end{center}
\end{table}

\begin{table}
 \caption{Recovery results of an approximately sparse signal with Crit.~$2$ as a stopping rule.}

\footnotesize
\begin{center} 
  \begin{tabular*}{0.7\textwidth} {@{\extracolsep{\fill}}l||c|c|c}
    Method & $ <\mathcal{N}_A> $ & $\min \mathcal{N}_A$  & $\max \mathcal{N}_A$\\
    \hline
    \nesta    & 18912 & 18773 & 19115\\
    \nesta + Ct     & 2667  & 2603  & 2713\\
    GPSR     & DNC& DNC & DNC\\
    GPSR + Ct	& DNC& DNC & DNC\\
    SpaRSA	& 10021	& 8353	& 12439\\
    SPGL1     & 1776  & 1073  & 2464\\
    FISTA     & 10724 & 10197 & 10980\\
    FPC Active Set     & DNC & DNC & DNC \\
    FPC Active Set (CG)      & DNC & DNC & DNC \\
    FPC      & DNC & DNC & DNC \\
    FPC-BB     & DNC & DNC & DNC \\
    Bregman-BB     & 2045  & 2045  & 2045
\end{tabular*}
\label{table14}
\end{center}
\end{table}

\section{An all-purpose algorithm}
\label{sec:flexible}

A distinguishing feature is that \nesta is able to cope with a wide
range of standard regularizing functions. In this section, we present
two examples: nonstandard $\ell_1$ minimization and total-variation
minimization.

\subsection{Nonstandard sparse reconstruction: $\ell_1$ analysis}
\label{sec:analysis}

Suppose we have a signal $x \in \mathbb{R}^n$, which is assumed to be
approximately sparse in a transformed domain such as the wavelet, the
curvelet or the time-frequency domains. Let $W$ be the corresponding
synthesis operator whose columns are the waveforms we use to
synthesize the signal $x = W \alpha$ (real-world signals do not admit
an exactly sparse expansion); e.g.~the columns may be wavelets,
curvelets and so on, or both. We will refer to $W^*$ as the analysis
operator.  As before, we have (possibly noisy) measurements $b = Ax^0
+ z$.  The {\em synthesis} approach attempts reconstruction by solving
\begin{equation}
\begin{array}{ll}
 \text{minimize}   & \quad \|\alpha\|_{\ell_1} \\
 \text{subject to} & \quad  \|b - A W \alpha\|_{\ell_2} \le \epsilon,
\end{array}
\label{eqn:l1syn}
\end{equation}
while the {\em analysis} approach solves the related problem
\begin{equation}
\begin{array}{ll}
 \text{minimize}   & \quad \|W^* x\|_{\ell_1} \\
 \text{subject to} & \quad  \|b - A x\|_{\ell_2} \le \epsilon.
\end{array}
\label{eq:l1an}
\end{equation}

If $W$ is orthonormal, the two problems are equivalent, but in
general, these give distinct solutions and current theory explaining
the differences is still in its infancy.  The article \cite{EladPrior}
suggests that synthesis may be overly sensitive, and argues with
geometric heuristics and numerical simulations that analysis is
sometimes preferable.

Solving $\ell_1$-analysis problems with \nesta is straightforward as
only Step $1$ needs to be adapted.  We have
\[
f_\mu(x) = \max_{x \mathcal{Q}_s} \langle u , W^*x \rangle -
\frac{\mu}{2}\|u\|_{\ell_2}^2,
\]
and the gradient at $x$ is equal to 
\[
\nabla f_\mu(x) = W u_\mu(x);
\]
here, $u_\mu(x)$ is given by
\[
(u_\mu(x))[i] = \begin{cases}
\mu^{-1}  (W^*x)[i], &  \text{if }   |(W^*x)[i]| < \mu,\\
\text{sgn}((W^*x)[i]), & \text{otherwise}.
\end{cases}
\]

Steps $2$ and $3$ remain unchanged. The computational complexity of
the algorithm is then increased by an extra term, namely $2 \,
\mathcal{C}_{W}$ where $\mathcal{C}_W$ is the cost of applying $W$ or $W^*$
to a vector. In practical situations, there is often a fast algorithm
for applying $W$ and $W^*$, e.g.~a fast wavelet transform
\cite{ima:mallat98}, a fast curvelet transform \cite{curv:demanet}, a
fast short-time Fourier transform \cite{ima:mallat98} and so on, which
makes this a low-cost extra step\footnote{The ability to solve the
analysis problem also means that \nesta can easily solve reweighted
$\ell_1$ problems \cite{Candes:2008le} with no change to the code.}.

\subsection{Numerical results for nonstandard $\ell_1$ minimization}

Because \nesta is one of very few algorithms that can solve both the
analysis and synthesis problems efficiently, we tested the performance
of both analysis and synthesis on a simulated real-world signal from
the field of radar detection.  The test input is a superposition of
three signals.  The first signal, which is intended to make recovery
more difficult for any smaller signals, is a plain sinusoid with amplitude
of 1000 and frequency near 835 MHz.  

A second signal, similar to a Doppler pulse radar, is at a carrier
frequency of 2.33 GHz with maximum amplitude of 10, a pulse width of 1
$\mu s$ and a pulse repetition interval of 10 $\mu$s; the pulse
envelope is trapezoidal, with a 10 ns rise time and 40 ns fall time.
This signal is more than 40 dB lower than the pure sinusoid, since the
maximum amplitude is 100$\times$ smaller, and since the radar is
nonzero only 10\% of the time.  The Doppler pulse was chosen to be
roughly similar to a realistic weather Doppler radar.  In practice,
these systems operate at 5 cm or 10 cm wavelengths (i.e.~6 or 3 GHz)
and send out short trapezoidal pulses to measure the radial velocity
of water droplets in the atmosphere using the Doppler effect.

The third signal, which is the signal of interest, is a
frequency-hopping radar pulse with maximum amplitude of 1 (so about 20
dB beneath the Doppler signal, and more than 60 dB below the
sinusoid).  For each instance of the pulse, the frequency is chosen
uniformly at random from the range 200 MHz to 2.4 GHz.  The pulse
duration is 2 $\mu s$ and the pulse repetition interval is 22 $\mu s$,
which means that some, but not all, pulses overlap with the Doppler
radar pulses.  The rise time and fall time of the pulse envelope are
comparable to the Doppler pulse.  Frequency-hopping signals may arise
in applications because they can be more robust to interference and
because they can be harder to intercept.  When the carrier frequencies
are not known to the listener, the receiver must be designed to cover
the entire range of possible frequencies (2.2 GHz in our case).  While
some current analog-to-digital converters (ADC) may be capable of
operating at 2.2 GHz, they do so at the expense of low precision.
Hence this situation may be particularly amenable to a compressed
sensing setup by using several slower (but accurate) ADC to cover a
large bandwidth.

We consider the exact signal to be the result of an infinite-precision
ADC operating at 5 GHz, which corresponds to the Nyquist rate for
signals with 2.5 GHz of bandwidth.  Measurements are taken using an
orthogonal Hadamard transform with randomly permuted columns, and
these measurements were subsequently sub-sampled by randomly choosing
$m = .3 n$ rows of the transform (so that we undersample Nyquist by
$10/3$).  Samples are recorded for $T = 209.7 \mu$s, which corresponds
to $n = 2^{20}$.  White noise was added to the measurements to make a
60 dB signal-to-noise ratio (SNR) (note that the effective SNR for the
frequency-hopping pulse is much lower).  The frequencies of the
sinusoid and the Doppler radar were chosen such that they were not
integer multiples of the lowest recoverable frequency $f_{min} =
1/(2T)$.

For reconstruction, the signal is analyzed with a tight frame of Gabor
atoms that is approximately 5.5$\times$ overcomplete.  The particular
parameters of the frame are chosen to give reasonable reconstruction,
but were not tweaked excessively.  It is likely that differences in
performance between analysis and synthesis are heavily dependent on
the particular dictionary.

To analyze performance, we restrict our attention to the frequency
domain in order to simplify comparisons.  The top plot in
Figure~\ref{fig:radar} shows the frequency components of the original,
noiseless signal.  The frequency hopping pulse barely shows up since
the amplitude is 1000$\times$ smaller than the sinusoid and since each
frequency only occurs for 1 $\mu$s (of 210 $\mu$s total).

The bottom plots in Figure~\ref{fig:radar} show the spectrum of the
recovered signal using analysis and synthesis, respectively.  For this
test, analysis does a better job at finding the frequencies belonging
to the small pulse, while synthesis does a better job recreating the
large pulse and the pure tone.  The two reconstructions used slightly
different values of $\mu$ to account for the redundancy in the size of
the dictionary; otherwise, algorithm parameters were the same.  In the
analysis problem, \nesta took 231 calls to the analysis/synthesis
operator (and 231 calls to the Hadamard transform); for synthesis,
\nesta took 1378 calls to the analysis/synthesis operator (and 1378 to
the Hadamard transform). With \nestansp, synthesis is more
computationally expensive than analysis since no change of variables
trick can be done; in the synthesis case, $W$ and $W^*$ are used in
Step $2$ and $3$ while in the analysis case, the same operators are
used once in Step $1$ (this is accomplished by the previously
mentioned change-of-variables for partial orthogonal measurements).

As emphasized in \cite{EladPrior}, when $W$ is overcomplete, the
solution computed by solving the analysis problems is likely to be
denser than in the synthesis case. In plain English, the analysis
solution may seem ``noisier'' than the synthesis solution. But the
compactness of the solution of the synthesis problem may also be its
weakness: an error on one entry of $\alpha$ may lead to a solution
that differs a lot. This may explain why the frequency-hopping radar
pulse is harder to recover with the synthesis prior.

Because all other known first-order methods solve only the synthesis
problem, \nesta may prove to be extremely useful for real-world
applications.  Indeed, this simple test suggests that analysis may
sometimes be much preferable to synthesis, and given a signal with
$2^{20}$ samples (too large for interior point methods), we know of no
other algorithm that can return the same results.

\begin{figure}
\begin{center}
\begin{centering}
\includegraphics[scale=0.55]{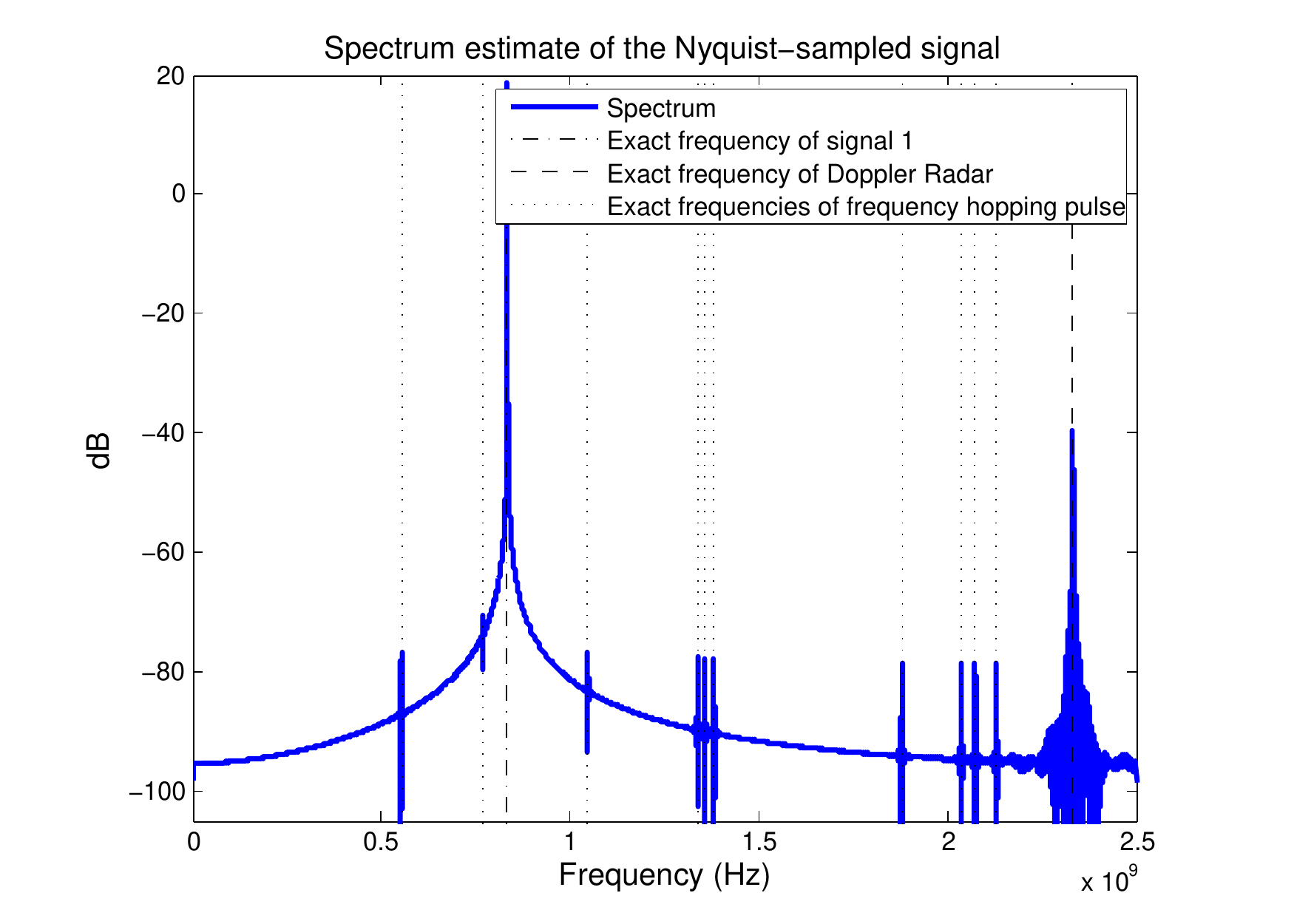}\\
\end{centering}
\vfill
\begin{centering}
\includegraphics[scale=0.55]{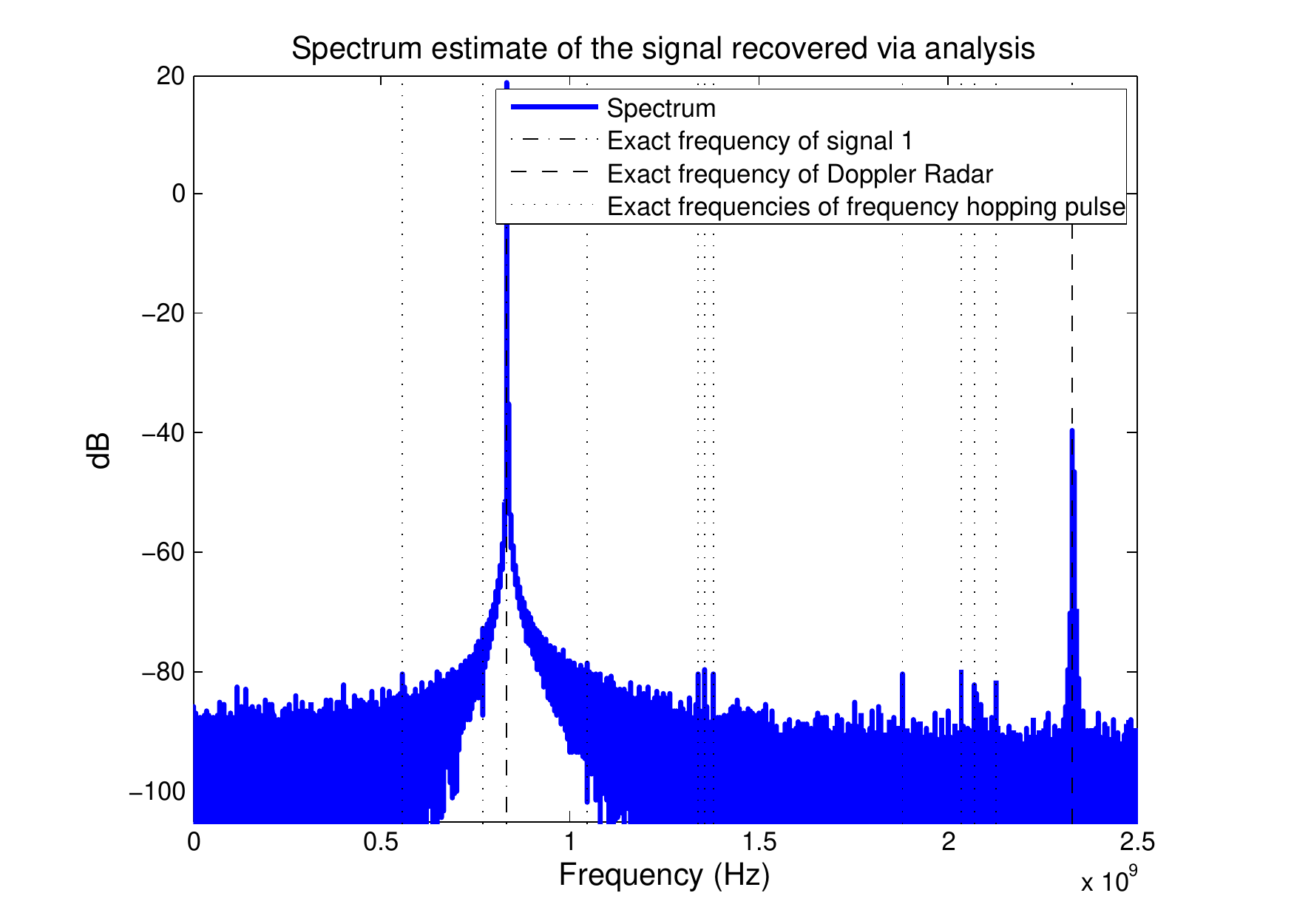}
\end{centering}
\vfill
\begin{centering}
\includegraphics[scale=0.55]{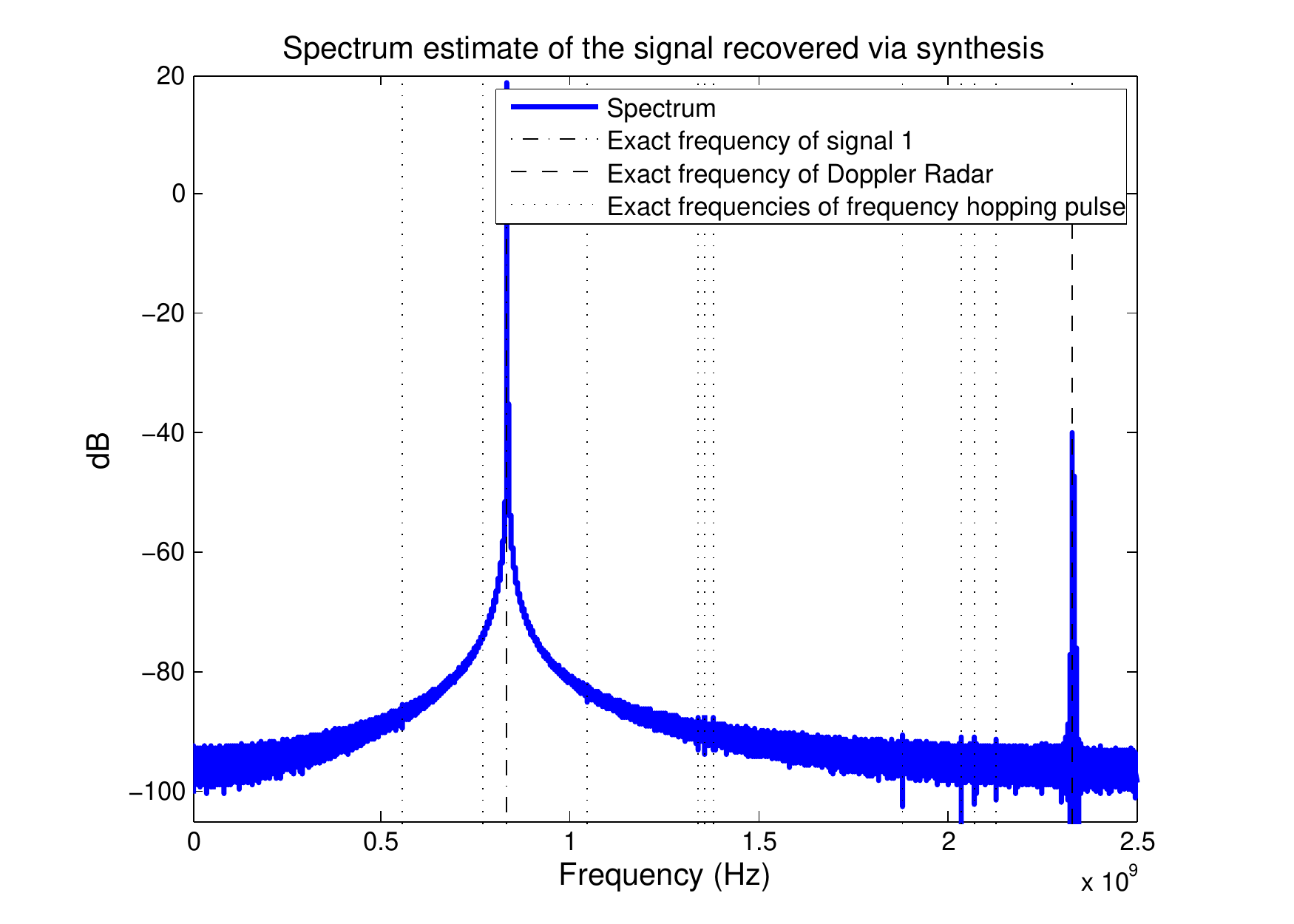}
\end{centering}

\caption{\textbf{Top:} spectrum estimate of the exact signal, no
noise.  The pure tone at 60 dB and the Doppler radar at 20 dB
dominate the 0 dB frequency hopping pulses. \textbf{Middle:}
spectrum estimate of the recovered signal using analysis prior, with
60 dB SNR.  The spectrum appears noisy, but the frequency hopping
pulses stand out. \textbf{Bottom:} spectrum estimate of the
recovered signal using synthesis prior, with 60 dB SNR.  The
spectrum appears cleaner, but the small 0 dB pulses do not appear.}
\label{fig:radar}
\end{center}
\end{figure}


\subsection{Total-variation minimization}\label{sec:tvmin}

Nesterov's framework also makes total-variation minimization
possible. The TV norm of a 2D digital object $x[i,j]$ is given by
\[
\|x\|_{TV} := \sum_{i,j} \|\nabla x[i,j]\|, \quad \nabla x [i,j]
= \begin{bmatrix} (D_1 x)[i,j]\\ (D_2 x)[i,j]\end{bmatrix},
\]
where $D_{1}$ and $D_2$ are the horizontal and vertical differences
\begin{align*}
(D_1 x)[i,j] & = x[i+1,j] - x[i,j], \\
(D_2 x)[i,j] & = x[i,j+1] - x[i,j].
\end{align*}
Now the TV norm can be expressed as follows: 
\begin{equation}\label{eq:tvnorm}
\|x\|_{TV} = \max_{u \mathcal{Q}_d} \langle u,Dx \rangle,
\end{equation}
where $u = [u_1,u_2]^* \in \mathcal{Q}_d$ if and only for each
$(i,j)$, $u^2_1[i,j] + u_2^2[i,j] \le 1$, and $D = [D_1,D_2]^*$.  The
key feature of Nesterov's work is to smooth a well-structured
nonsmooth function as follows (notice in \eqref{eq:tvnorm} the
similarity between the TV norm and the $\ell_1$ norm):
\[
\max_{u \mathcal{Q}_d} \langle u,Dx \rangle - \mu p_d(u).
\]
Choosing $p_d(u) = \frac{1}{2} \|u\|_{\ell_2}^2$ provides a reasonable
prox-function that eases the computation of $\nabla f_\mu$. Just as
before, changing the regularizing function only modifies Step $1$ of
\nestansp. Here,
\[
f_\mu(x) = \max_{u \in \mathcal{Q}_d} \langle u , D x \rangle -
\frac{\mu}{2}\|u\|_{\ell_2}^2.
\]
Then as usual,
\[
\nabla f_\mu(x) = D^* u_\mu(x), 
\]
where $u_\mu(x)$ is of the form $[u_1,u_2]^*$ and for each $a \in
\{1,2\}$, 
\[
u_a[i,j] = \begin{cases}
\mu^{-1} (D_a x)[i,j], & \text{if }    \|\nabla x [i,j]\| < \mu, \\
\|\nabla x [i,j]\|^{-1} (D_a x)[i,j], & \text{otherwise}.
\end{cases}
\]
The application of $D$ and $D^*$ leads to a negligible computational
cost (sparse matrix-vector multiplications).

\subsection{Numerical results for TV minimization}

We are interested in solving
\begin{equation}
\begin{array}{ll}
\text{minimize}   & \quad \|x\|_{TV} \\
\text{subject to} & \quad  \|b - A x\|_{\ell_2} \le \epsilon.
\end{array}
\label{eq:tvmin}
\end{equation}
To be sure, a number of efficient TV-minimization algorithms have been
proposed to solve \eqref{eq:tvmin} in the special case $A = I$
(denoising problem), see \cite{rest:chambolle04,DS:05,GO:08}.  In
comparison, only a few methods have been proposed to solve the more
general problem \eqref{eq:tvmin} even when $A$ is a projector. Known
methods include interior point methods ($\ell_1$-magic)
\cite{l1magic}, proximal-subgradient methods \cite{Twist07,CP:08},
Split-Bregman \cite{GO:08}, and the very recently introduced
RecPF\footnote{available at
\url{http://www.caam.rice.edu/~optimization/L1/RecPF/}.}
\cite{RecPF}, which operates in the special case of partial Fourier
measurements.  Roughly, proximal gradient methods approach the
solution to \eqref{eq:tvmin} by iteratively updating the current
estimate $x_k$ as follows:
\[
x_{k+1} = \text{Prox}_{TV,\gamma}\left(x_k + \alpha A^*(b - Ax_k) \right),
\]
where $\mbox{Prox}_{TV,\gamma}$ is the proximity operator of TV, see
\cite{CombettesWajs05} and references therein,
\[
\text{Prox}_{TV,\gamma}(z) = \Argmin{x} \gamma \|x\|_{TV}
+\frac{1}{2}\|x - z \|_{\ell_2}^2.
\]
Evaluating the proximity operator at $z$ is equivalent to solving a TV
denoising problem. In \cite{Twist07}, the authors advocate the use of
a side algorithm (for instance Chambolle's algorithm
\cite{rest:chambolle04}) to evaluate the proximity operator. There are
a few issues with this approach. The first is that side algorithms
depend on various parameters, and it is unclear how one should select
them in a robust fashion. The second is that these algorithms are
computationally demanding which makes them hard to apply to
large-scale problems.

To be as fair as possible, we decided to compare \nesta with
algorithms for which a code has been publicly released; this is the
case for the newest in the family, namely, RecPF (as $\ell_1$-magic is
based on an interior point method, it is hardly applicable to this
large-scale problem). Hence, we propose comparing \nesta for TV
minimization with RecPF. 

Evaluations are made by comparing the performances of \nesta (with
continuation) and RecPF on a set of images composed of random
squares. As in Section~\ref{sec:numeric}, the dynamic range of the
signals (amplitude of the squares) varies in a range from $20$ to $40$
dB.  The size of each image $x$ is $1024 \times 1024$; one of these
images is displayed in the top panel of Figure~\ref{fig:squares}. The
data $b$ are partial Fourier measurements as in \cite{CRT:cs}; the
number of measurements $m = n/10$. White Gaussian noise of standard
deviation $\sigma=0.1$ is added. The parameters of \nesta are set up
as follows:
\[
x_0 = A^*b, \quad \mu = 0.2, \quad \delta = 10^{-5}, \quad \CtSteps= 5,
\]
and the initial value of $\mu$ is
\[
\mu_0 = 0.9 \, \max_{ij} \|\nabla x_0[i,j]\|.
\]
The maximal number of iterations is set to $\mathcal{I}_{\max} =
4,\!000$. As it turns out, TV minimization from partial Fourier
measurements is of significant interest in the field of Magnetic
Resonance Imaging \cite{lustig07}.

As discussed above, RecPF has been designed to solve TV minimization
reconstruction problems from partial Fourier measurements. We set the
parameters of RecPF to their default values except for the parameter
\texttt{tol\_rel\_inn} that is set to $10^{-5}$. This choice makes
sure that this converges to a solution close enough to \nestansp's
output.  Figure~\ref{fig:squares} shows the the solution computed by
RecPF (bottom left) and \nesta (bottom
right). 

\begin{figure}
\begin{center}
\includegraphics[scale=0.25]{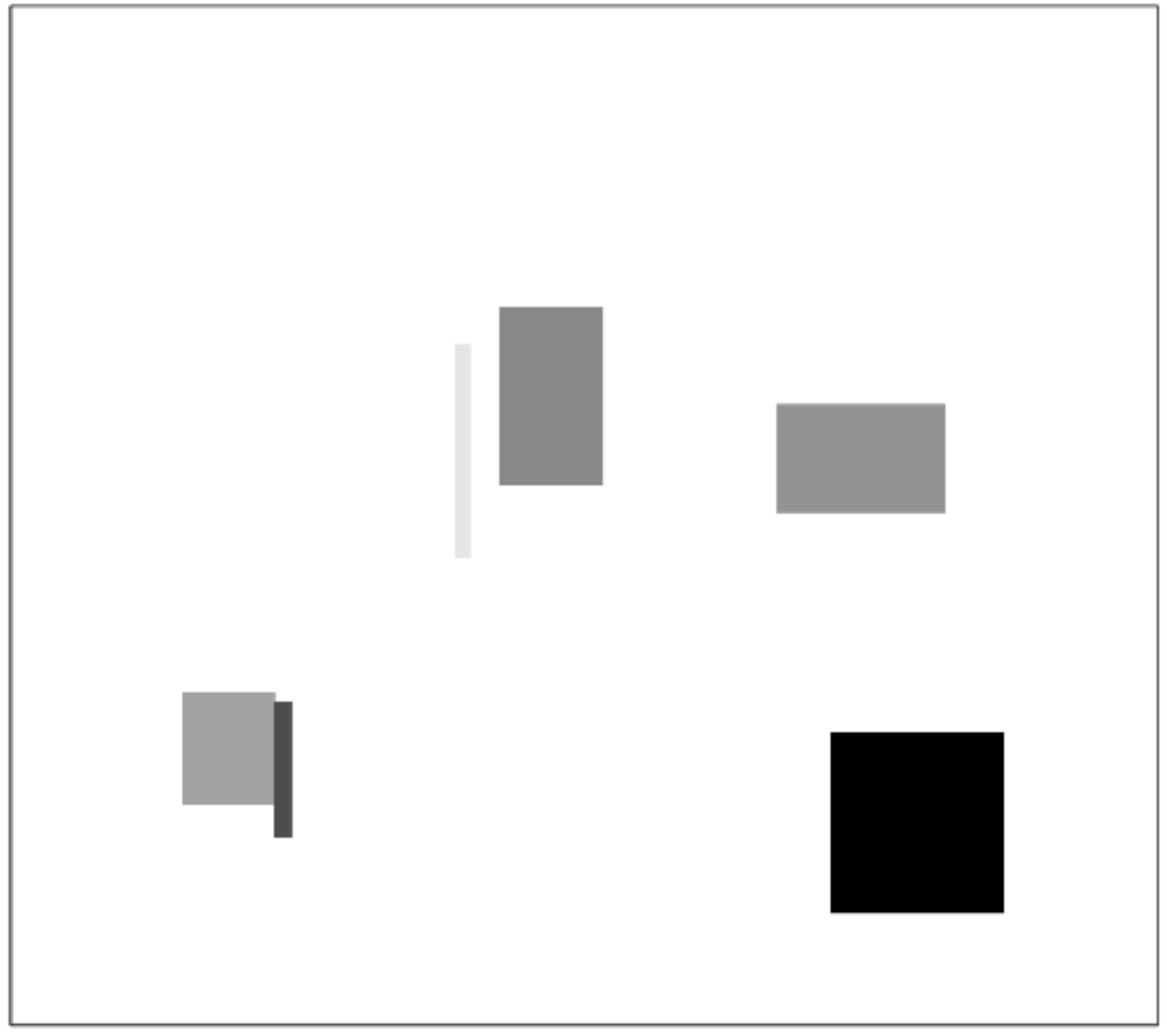}\\
\vspace{1cm}
{\includegraphics[scale=0.25]{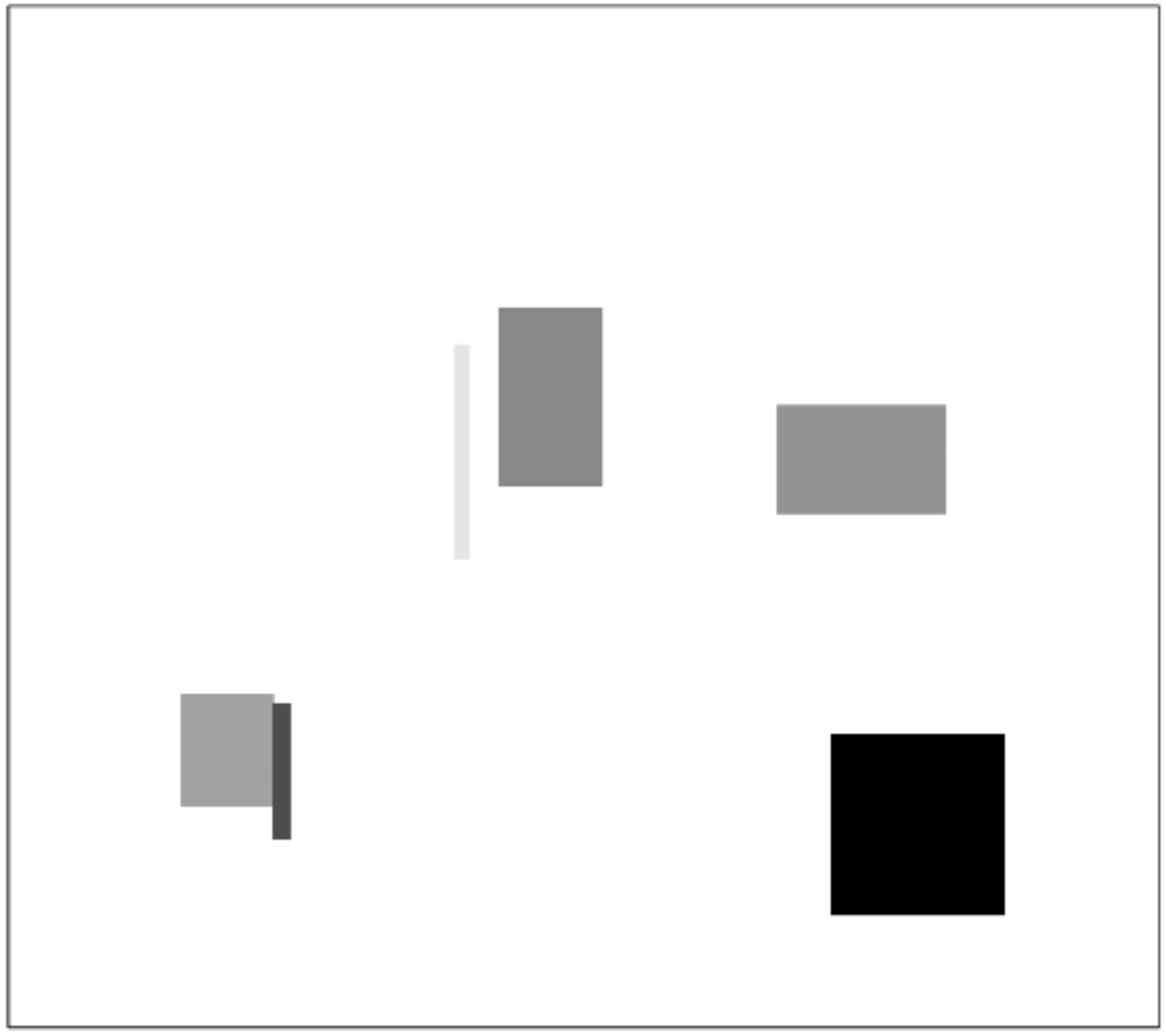}
\hspace{1cm}
\includegraphics[scale=0.25]{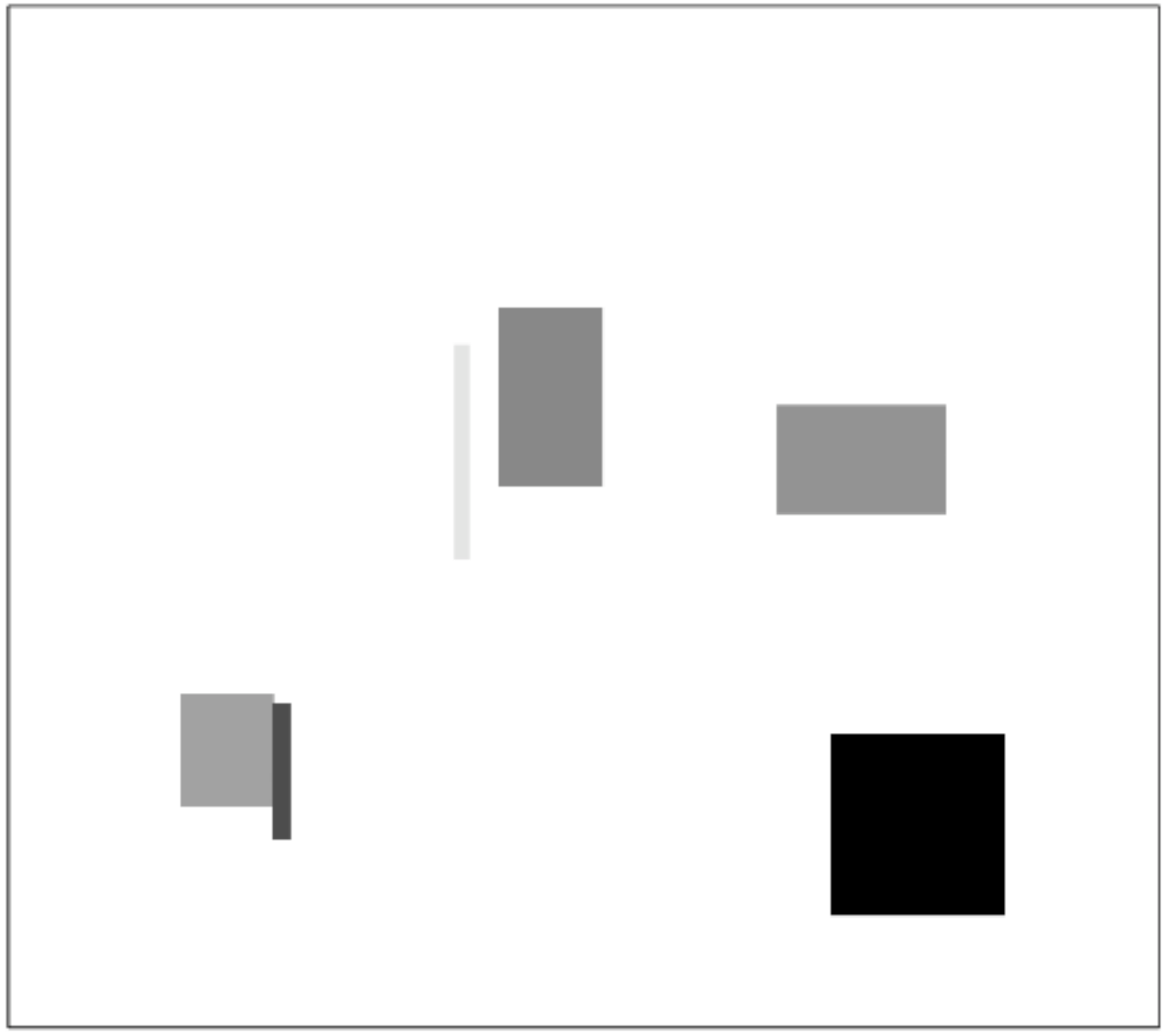}
}
\end{center}
\caption{\textbf{Top:} original image of size $1024 \times$ 1024 with
a dynamic range of about $40$ dB. \textbf{Bottom-Left:} RecPF
solution. \textbf{Bottom-Right:} \nesta solution.}
\label{fig:squares}
\end{figure}

The curves in Figure~\ref{fig:tverror} show the number of calls to $A$
or $A^*$; mid-points are averages over $5$ random trials, with error
bars indicating the minimum and maximum number of calls. Here, RecPF
is stopped when
\begin{align*}
 \|x_{\text{RecPF}}\|_{TV} & \leq 1.05 \|x_N \|_{TV},\\
 \|b - Ax_{\text{RecPF}}\|_{\ell_2} & \leq 1.05 \|b - Ax_N
 \|_{\ell_2},
\end{align*}
where $x_N$ is the solution computed via \nestansp.  As before
continuation is very efficient when the dynamic range is high
(typically higher than $40$ dB). An interesting feature is that the
numbers of calls are very similar over all five trials. When the
dynamic range increases, the computational costs of both \nesta and
RecPF naturally increase. Note that in the $60$ and $80$ dB
experiments, RecPF did not converge to the solution and this is the
reason why the number of calls saturates. While both methods have a
similar computational cost in the low-dynamic range regime, \nesta has
a clear advantage in the higher-dynamic range regime. Moreover, the
number of iterations needed to reach convergence with \nesta with
continuation is fairly low---300-400 calls to $A$ and $A^*$---and so
this algorithm is well suited to large scale problems.

\begin{figure}
\begin{center}
\includegraphics[scale=0.55]{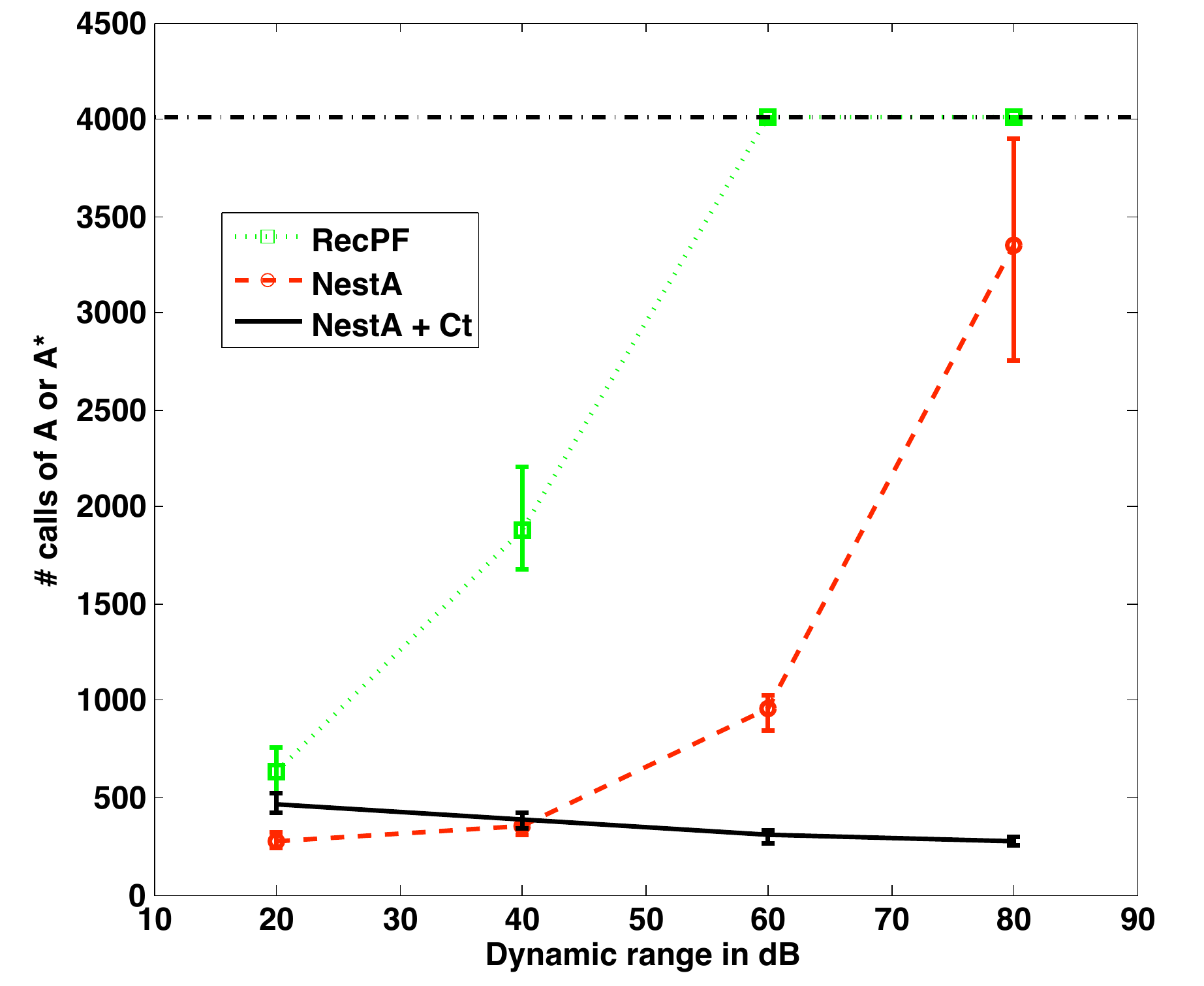}
\caption{Number of calls to $A$ and $A^*$ as a function of the
 dynamic range. \textit{Solid line:} \nesta with
 continuation. \textit{Dashed line:} \nestansp. \textit{Dotted line:}
 RecPF. \textit{Dash-dotted:} maximum number of iterations. In the
 $60$ and $80$ dB experiments, RecPF did not converge to the
 solution and this is the reason why the number of calls
 saturates.}
\label{fig:tverror}
\end{center}
\end{figure}


\section{Discussion}\label{sec:discussion}

In this paper, we have proposed an algorithm for general sparse
recovery problems, which is based on Nesterov's method.  This
algorithm is accurate and competitive with state-of-the-art
alternatives. In fact, in applications of greatest interest such as
the recovery of approximately sparse signals, it outperforms most of
the existing methods we have used in our comparisons and is comparable
to the best. Further, what is interesting here, is that we have not
attempted to optimize the algorithm in any way. For instance, we have
not optimized the parameters $\{\alpha_k\}$ and $\{\tau_k\}$, or the
number of continuation steps as a function of the desired accuracy
$\delta$, and so it is expected that finer tuning would speed up the
algorithm. Another advantage is that \nesta is extremely flexible in
the sense that minor adaptations lead to efficient algorithms for a
host of optimization problems that are crucial in the field of
signal/image processing.

\subsection{Extensions}

This paper focused on the situation in which $A^* A$ is a projector
(the rows of $A$ are orthonormal). This stems from the facts that 1)
the most computationally friendly compressed sensing are of this form,
and 2) it allows fast computations of the two sequence of iterates
$\{y_k\}$ and $\{z_k\}$. It is important, however, to extend \nesta as
to be able to cope with a wider range of problem in which $A^*A$ is
not a projection (or not diagonal).

In order to do this, observe that in Steps $2$ and $3$, we need to
solve problems of the form
\[
y_k = \Argmin{x \in \mathcal{Q}_p} \| x -q\|_{\ell_2}^2, 
\]
for some $q$, and we have seen that the solution is given by $y_k =
\mathcal{P}_{\mathcal{Q}_p}(q)$, where $\mathcal{P}_{\mathcal{Q}_p}$
is the projector onto $\mathcal{Q}_p := \{x : \|Ax - b\|_{\ell_2} \le
\epsilon\}$. The solution is given by 
\begin{equation}
\label{eq:projQ}
y_k = (I+ \lambda A^* A)^{-1}(q + \lambda A^*b)
\end{equation}
for some $\lambda \ge 0$. When the eigenvalues of $A^*A$ are well
clustered, the right-hand side of \eqref{eq:projQ} can be computed
very efficiently via a few conjugate gradients (CG) steps. Note that
this is of direct interest in compressed sensing applications in which
$A$ is a random matrix since in all the cases we are familiar with,
the eigenvalues of $A^*A$ are tightly clustered. Hence, \nesta may be
extended to general problems while retaining its efficiency, with the
proviso that a good rule for selecting $\lambda$ in \eqref{eq:projQ}
is available; i.e.~such that $\|Ay_k - b\|_{\ell_2} = \epsilon$ unless
$q \in \mathcal{Q}_p$.  Of course, one can always eliminate the
problem of finding such a $\lambda$ by solving the unconstrained
problem \QPL instead of $(\text{BP}_{\epsilon})$. In this case, each
\nesta iteration is actually very cheap, no matter how $A$ looks like.

Finally, we also observe that Nesterov's framework is likely to
provide efficient algorithms for related problems, which do not have
the special $\ell_1 + \ell_2^2$ structure. One example might be the
Dantzig selector, which is a convenient and flexible estimator for
recovering sparse signals from noisy data \cite{DS}:
\begin{equation}
\label{eq:DS} 
\begin{array}{ll}
 \text{minimize}   & \quad \|x\|_{\ell_1}\\
 \text{subject to} & \quad  \|A^*(b-AX)\|_{\ell_\infty} \le \delta. 
\end{array}
\end{equation}
This is of course equivalent to the unconstrained problem
\[
\text{minimize} \quad \lambda \|x\|_{\ell_1} +
\|A^*(b-AX)\|_{\ell_\infty}
\]
for some value of $\lambda$. Clearly, one could apply Nesterov's
smoothing techniques to smooth both terms in the objective functional
together with Nesterov's accelerated gradient techniques, and derive a
novel and efficient algorithm for computing the solution to the
Dantzig selector. This is an example among many others. Another might
be the minimization of a sum of two norms, e.g.~an $\ell_1$ and a TV
norm, under data constraints.

\subsection{Software}

In the spirit of reproducible research \cite{RS08}, a Matlab version
of \nesta will be made available at: \url{http://www.acm.caltech.edu/~nesta/}

\subsection*{Acknowledgements}

S.~Becker wishes to thank Peter Stobbe for the use of his Hadamard
Transform and Gabor frame code, and Wotao Yin for helpful discussions
about RecPF. J.~Bobin wishes to thank Hamza Fawzi for fruitful
discussions, and E.~Cand\`es would like to thank Jalal Fadili for his
suggestions. We are grateful to Stephen Wright for his comments on an
earlier version of this paper, for suggesting to use a better
version of GPSR, and encouraging us to test SpaRSA. Thanks Stephen!

\bibliographystyle{siam}
\bibliography{NESTA}

\end{document}